\documentclass[12pt]{amsart}
\usepackage[cmtip,all]{xy}
\usepackage{graphicx}
\usepackage{amssymb}
\usepackage{mathrsfs}
\usepackage{amsfonts}
\usepackage{amsmath}

\newtheorem{theorem}{Theorem}[section]
\newtheorem{lemma}[theorem]{Lemma}
\newtheorem{proposition}[theorem]{Proposition}
\newtheorem{corollary}[theorem]{Corollary}

\theoremstyle{definition}

 \theoremstyle{remark}
\newtheorem{remark}[theorem]{Remark}

 \numberwithin{equation}{section}

\begin{document}

\title[sharp Sobolev trace inequalities for higher order derivatives]{sharp Sobolev trace inequalities for higher order derivatives}


\author{Qiaohua  Yang}
\address{School of Mathematics and Statistics, Wuhan University, Wuhan, 430072, People's Republic of China}

\email{qhyang.math@gmail.com; qhyang.math@whu.edu.cn}

\thanks{The work was partially supported by the National Natural
Science Foundation of China(No.11201346).}






\begin{abstract}
Motivated by a recent work of Ache and Chang concerning the sharp Sobolev trace inequality and  Lebedev-Milin inequalities of order four
on the Euclidean unit ball,
we derive such inequalities on the Euclidean unit ball for higher order derivatives.
 By using,  among other things,  the scattering
theory on hyperbolic spaces and the generalized
Poisson kernel, we obtain the explicit formulas of
extremal functions of such inequations. Moreover, we also derive the sharp  trace Sobolev inequalities on half spaces
for higher order derivatives.
  Finally, we  compute the explicit formulas of
 adapted
metric, introduced by Case and Chang,  on the Euclidean unit ball,  which is of independent interest.
   \end{abstract}

\maketitle


\section{Introduction}
It is well known that the Sobolev inequalities and sharp constants play  an important role in problems in analysis and conformal geometry.
An elementary example is the following  classical Sobolev inequality on the standard sphere $(\mathbb{S}^{n},g_{\mathbb{S}^{n}})$ for $n\geq3$:
\begin{equation}\label{1.1}
\begin{split}
\left(\frac{1}{\omega_{n}}\int_{\mathbb{S}^{n}}|f|^{\frac{2n}{n-2}}d\sigma\right)^{\frac{n-2}{n}}\leq\frac{4}{n(n-2)\omega_{n}}
\int_{\mathbb{S}^{n}}|\widetilde{\nabla}f|^{2}d\sigma+\frac{1}{\omega_{n}}\int_{\mathbb{S}^{n}}|f|^{2}d\sigma,
\end{split}
\end{equation}
where $d\sigma$ is Lebesgue measure on $\mathbb{S}^{n}$, $\omega_{n}=\frac{2\pi^{\frac{n+1}{2}}}{\Gamma(\frac{n+1}{2})}$  is the volume of
$\mathbb{S}^{n}$ and $\widetilde{\nabla}$ is the sphere  gradient on $(\mathbb{S}^{n},g_{\mathbb{S}^{n}})$.
Using the conformal invariance, one observes inequality (\ref{1.1}) is equivalent to the sharp Sobolev inequality on $\mathbb{R}^{n}$
given as follows
\begin{equation}\label{1.2}
\begin{split}
\frac{\Gamma(\frac{n+2}{2})}{\Gamma(\frac{n-2}{2})}\omega^{\frac{2}{n}}_{n}\left(
\int_{\mathbb{R}^{n}}|f|^{\frac{2n}{n-2}}dx\right)^{\frac{n-2}{n}}\leq
\int_{\mathbb{R}^{n}}|\nabla f|^{2}dx,
\end{split}
\end{equation}
where $\nabla f$ is the gradient of $f$ with respect to the Euclidean metric.
In a limiting case, namely $n=2$, (\ref{1.1}) takes  the form of  Moser-Onfori
inequality (\cite{mo,on})
  \begin{equation}\label{1.3}
  \begin{split}
\log \left(\frac{1}{4\pi}\int_{\mathbb{S}^{2}}e^{f}d\sigma\right)\leq
\frac{1}{16\pi}
\int_{\mathbb{S}^{2}}|\widetilde{\nabla}f|^{2}d\sigma+\frac{1}{4\pi}\int_{\mathbb{S}^{2}}|f|^{2}d\sigma.
\end{split}
\end{equation}
Inequality (\ref{1.3}) has been widely used in   analysis and conformal geometry, in particular,
in the problem of prescribing Gaussian curvature on the sphere (see \cite{ch,os1,os2,os3}).

Another example is the Sobolev trace inequality.  Denote by $\mathbb{B}^{n+1}$ the unit ball on Euclidean space $\mathbb{R}^{n+1}$ with $\mathbb{S}^{n}$ as the boundary.
The Sobolev trace inequality on $\mathbb{B}^{n+1}$ reads as follow (see \cite{es1}): for $f\in C^{\infty}(\mathbb{S}^{n})$ and $n\geq 2$,
  \begin{equation}\label{1.4}
  \begin{split}
\frac{n-1}{2}\omega^{\frac{1}{n}}_{n}\left(\int_{\mathbb{S}^{n}}|f|^{\frac{2n}{n-1}}d\sigma\right)^{\frac{n-1}{n}}\leq
\int_{\mathbb{B}^{n+1}}|\nabla v|^{2}dx+\frac{n-1}{2}\int_{\mathbb{S}^{n}}|f|^{2}d\sigma,
\end{split}
\end{equation}
where $v$ is  a smooth extension of $f$ to $\mathbb{B}^{n+1}$. By
the conformal invariance, one observes inequality (\ref{1.4}) is also equivalent to the sharp Sobolev inequality on half space
(see \cite{es1}):
\begin{equation}\label{1.5}
\begin{split}
\frac{\Gamma(\frac{n+1}{2})}{\Gamma(\frac{n-1}{2})}\omega^{\frac{1}{n}}_{n}\left(
\int_{\mathbb{R}^{n}}|U(x,0)|^{\frac{2n}{n-1}}dx\right)^{\frac{n-1}{n}}\leq
\int_{\mathbb{R}_{+}^{n+1}}|\nabla U(x,y)|^{2}dxdy,
\end{split}
\end{equation}
where $\mathbb{R}_{+}^{n+1}=\mathbb{R}^{n}\times\mathbb{R}_{+}=\{(x,y): x\in\mathbb{R}^{n},\;y>0\}$.
In a limiting case, namely $n=1$, (\ref{1.4}) becomes the classical   Lebedev-Milin  inequality (\cite{l})
  \begin{equation}\label{1.6}
  \begin{split}
\log \left(\frac{1}{\pi}\int_{\mathbb{S}^{1}}e^{f}d\sigma\right)\leq
\frac{1}{4\pi}
\int_{\mathbb{B}^{2}}|\nabla v|^{2}dx+\frac{1}{\pi}\int_{\mathbb{S}^{1}}|f|^{2}d\sigma.
\end{split}
\end{equation}
Such Sobolev trace inequalities and Lebedev-Milin  inequality have also been widely used in analysis and geometry,
such as Yamabe problem
on manifolds with boundary (see \cite{es2}), the Bieberbach
conjecture (\cite{d}) and the compactness of isospectral planar domains (see \cite{os1,os2,os3}).

Notice that the operators involved in (\ref{1.1})-(\ref{1.6}), either the Laplace operator or
the conformal Laplace operator, are order two. Recently, the role played by
these  operators of order two has been  extended to operators of higher order, such as Paneiz operator, poly-Laplacian and GJMS operator.
In particular, Beckner \cite{be} established
 the higher order Sobolev inequality on  the standard sphere $(\mathbb{S}^{n},g_{\mathbb{S}^{n}})$. We state it as follow
\begin{theorem} [Beckner] \label{be1}
Let $\widetilde{\Delta}$ be the  Laplace-Beltrame operator on the standard sphere $(\mathbb{S}^{n},g_{\mathbb{S}^{n}})$
and  define
 \begin{equation}\label{1.7}
  \begin{split}
\mathcal{P}_{2\gamma}=\frac{\Gamma(B+\frac{1}{2}+\gamma)}{\Gamma(B+\frac{1}{2}-\gamma)},\;\;B=\sqrt{-\widetilde{\Delta}+\frac{(n-1)^{2}}{4}}.
\end{split}
\end{equation}
Then for $0<\gamma<\frac{n}{2}$,
\begin{equation}\label{1.8}
\begin{split}
\frac{\Gamma(\frac{n+2\gamma}{2})}{\Gamma(\frac{n-2\gamma}{2})}\omega^{\frac{2\gamma}{n}}_{n}\left(
\int_{\mathbb{S}^{n}}|f|^{\frac{2n}{n-2\gamma}}d\sigma\right)^{\frac{n-2\gamma}{n}}\leq
\int_{\mathbb{S}^{n}}f \mathcal{P}_{2\gamma}fd\sigma.
\end{split}
\end{equation}
Equality holds only for functions of the form
$$c|1-\langle\zeta,\xi\rangle|^{\frac{2\gamma-n}{2}},\;\;c\in \mathbb{R},\; \zeta\in \mathbb{B}^{n+1},\;\;\xi\in
\mathbb{S}^{n}.$$
If $\gamma=\frac{n}{2}$, then
\begin{equation}\label{1.9}
\begin{split}
\ln\left(
\frac{1}{\omega_{n}}\int_{\mathbb{S}^{n}}e^{f-\overline{f}}d\sigma\right)\leq\frac{1}{2n!\omega_{n}}
\int_{\mathbb{S}^{n}}f \mathcal{P}_{n}fd\sigma,
\end{split}
\end{equation}
where $\overline{f}=\frac{1}{\omega_{n}}\int_{\mathbb{S}^{n}}fd\sigma$ is the  integral average of $f$ on $\mathbb{S}^{n}$. Equality holds only for functions of the form
$$-n\ln|1-\langle\zeta,\xi\rangle|+c,\;\;\; \zeta\in \mathbb{B}^{n+1},\; \xi\in
\mathbb{S}^{n},\; c\in \mathbb{R}.$$
\end{theorem}

By
the conformal invariance, one observes inequality (\ref{1.8}) is  equivalent to the following sharp Sobolev inequality on
$\mathbb{R}^{n}$ (\cite{lie}, see also \cite{co}):
\begin{theorem}\label{be2}
Let $0<\gamma<\frac{n}{2}$. Then
\begin{equation}\label{1.10}
\begin{split}
\frac{\Gamma(\frac{n+2\gamma}{2})}{\Gamma(\frac{n-2\gamma}{2})}\omega^{\frac{2\gamma}{n}}_{n}\left(
\int_{\mathbb{R}^{n}}|f|^{\frac{2n}{n-2\gamma}}dx\right)^{\frac{n-2\gamma}{n}}\leq
\int_{\mathbb{R}^{n}}|(-\Delta)^{\gamma}f |^{2}dx.
\end{split}
\end{equation}
Equality holds only for functions of the form
$$c(\lambda^{2}+|x-x_{0}|^{2})^{-\frac{n-2\gamma}{2}}, x\in\mathbb{R}^{n},$$
where $c\in\mathbb{R}$, $\lambda>0$ and $x_{0}$ is some point in $\mathbb{R}^{n}$.
\end{theorem}

Recently, Ache and Chang \cite{ac} established sharp trace Sobolev inequality 
of order four on $\mathbb{B}^{n+1}$ for $n\geq3$. As an aplication, they used this inequality
 to characterize the extremal metric
of the main term in the log-determinant formula corresponding to the conformal
Laplacian coupled with the boundary Robin operator on $\mathbb{B}^{4}$ (see \cite{chq1,chq2}). We state the results as follow.
\begin{theorem} [Ache and Chang] \label{th1.1}
Given $f\in C^{\infty}(\mathbb{S}^{n})$ with $n>3$, suppose $v$ is a smooth extension of $f$ to the unit
ball $\mathbb{B}^{n+1}$ which also satisfies the Neumann boundary condition
  \begin{equation}\label{1.11}
  \begin{split}
\left.\frac{\partial v}{\partial n}\right|_{\mathbb{S}^{n}}=-\frac{n-3}{2}f.
\end{split}
\end{equation}
Then we have the inequality
  \begin{equation}\label{1.12}
  \begin{split}
2\frac{\Gamma(\frac{n+3}{2})}{\Gamma(\frac{n-3}{2})}\omega^{\frac{3}{n}}_{n}&\left(\int_{\mathbb{S}^{n}}|f|^{\frac{2n}{n-3}}d\sigma\right)^{\frac{n-3}{n}}
\\ &\leq
\int_{\mathbb{B}^{n+1}}|\Delta v|^{2}dx+2\int_{\mathbb{S}^{n}}|\widetilde{\nabla}f|^{2}d\sigma+\frac{(n+1)(n-3)}{2}\int_{\mathbb{S}^{n}}|f|^{2}d\sigma,
\end{split}
\end{equation}
where $\Delta f$ is the Laplacian of $f$ with respect to the Euclidean metric.
Moreover, equality holds if and only if $v$ is a biharmonic
extension of a function of the form $c|1-\langle z_{0},\xi\rangle|^{\frac{3-n}{4}}$, where $c$ is a constant, $\xi\in\mathbb{S}^{n}$,
$z_{0}$ is some point in $\mathbb{B}^{n+1}$, and $v$ satisfies the boundary condition (\ref{1.11}).
When $f=1$, inequality (\ref{1.12}) is
attained by the function $v(x)=1+\frac{n-3}{4}(1-|x|^{2})$.
\end{theorem}

\begin{theorem} [Ache and Chang] \label{th1.2}
Given $f\in C^{\infty}(\mathbb{S}^{3})$, suppose $v$ is a smooth extension of $f$ to the unit
ball $\mathbb{B}^{4}$ which also satisfies the Neumann boundary condition
  \begin{equation}\label{1.13}
  \begin{split}
\left.\frac{\partial v}{\partial n}\right|_{\mathbb{S}^{3}}=0.
\end{split}
\end{equation}
Then we have the inequality
  \begin{equation}\label{1.14}
  \begin{split}
\log \left(\frac{1}{2\pi^{2}}\int_{\mathbb{S}^{3}}e^{3(f-\overline{f})}d\sigma\right)\leq
\frac{3}{16\pi^{3}}
\int_{\mathbb{B}^{4}}|\Delta v|^{2}dx+\frac{3}{8\pi^{2}}\int_{\mathbb{S}^{3}}|\widetilde{\nabla}f|^{2}d\sigma.
\end{split}
\end{equation}
Moreover, equality holds if and only if $v$ is a biharmonic
extension of a function of the form $-\log|1-\langle z_{0},\xi\rangle|+c$, where $c$ is a constant, $\xi\in\mathbb{S}^{3}$,
$z_{0}$ is some point in $\mathbb{B}^{4}$ and $v$ satisfies the boundary condition (\ref{1.13}).
\end{theorem}

The proof of  Theorem \ref{th1.1} and \ref{th1.2} relies on the use of scattering
theory on hyperbolic space $(\mathbb{B}^{n+1},g_{\mathbb{B}})$, where $g_{\mathbb{B}}=\frac{4}{(1-|x|^{2})^{2}}g_{0}$
and $g_{0}=|dx|^{2}$ is the
Euclidean metric, and the right choice of distance function and adapted metric.
We remark that the  adapted
metric, introduced by Case and Chang \cite{cac},  is a  `natural' metric in the study of Sobolev  inequalities (see \cite{ac,chw,chy}).
The explicit formulas of
 adapted
metric is computed by
Ache and Chang \cite{ac}  only  in the case $\gamma=\frac{1}{2}$ and $\gamma=\frac{3}{2}$. To the best of our knowledge, the  explicit formulas of
 adapted
metric is unknown for the rest of cases.

Very recently, Ng\^o, Nguyen and Pham \cite{n} show that (\ref{1.12}) is  equivalent to the following sharp Sobolev trace inequality on half space
via M\"obius transform:

\begin{theorem} \label{th1.3}
Let $U\in W^{2,2}(\overline{\mathbb{R}_{+}^{n+1}})$ be
satisfied the Neumann boundary condition
  \begin{equation}\label{1.15}
  \begin{split}
\partial_{y}U(x,0)=0,
\end{split}
\end{equation}
Where $W^{2,2}(\overline{\mathbb{R}_{+}^{n+1}})$ is the usually Sobolev space. Then we have the sharp trace inequality
  \begin{equation}\label{1.16}
  \begin{split}
2\frac{\Gamma(\frac{n+3}{2})}{\Gamma(\frac{n-3}{2})}\omega^{\frac{3}{n}}_{n}\left(\int_{\mathbb{R}^{n}}|U(x,0)|^{\frac{2n}{n-3}}dx\right)^{\frac{n-3}{n}}
\leq
\int_{\mathbb{R}_{+}^{n+1}}|\Delta U(x,y)|^{2}dxdy.
\end{split}
\end{equation}
Furthermore, equality in (\ref{1.10}) holds if and only
if $U$ is a biharmonic extension of a function of the form $c(1+|x-x_{0}|^{2})^{-(n-3)/2}$, where $c$ is a constant,
 $x\in\mathbb{R}^{n}$,
$x_{0}$ is some fixed point in $\mathbb{R}^{n}$ and $U$ fulfills the
boundary condition (\ref{1.15}).
\end{theorem}
In the same paper, Ng\^o, Nguyen and Pham \cite{n}, among other results, propose a slightly different approach to
prove Sobolev trace inequality of order six, while
 Case and Luo \cite{cal} obtained, among other results, the same sharp Sobolev trace inequalities
by deeply work of on the boundary operators. However, it seems that the argument in \cite{n,cal} would become increasingly delicate when the order of the operator is large.
A clear question   is   ``What is the situation for sharp Sobolev trace inequality of higher order?" Another question is
    `` what is
 the explicit formula of
extremal function of such inequations?"  In this paper, we shall give the answer  of both  questions.

The main results of this paper are the following three theorems.
\begin{theorem} \label{th1.6}Let $n>3$ and $m\geq1$ with $2m+1<n$.
Given $f\in C^{\infty}(\mathbb{S}^{n})$ , suppose $v$ is a smooth extension of $f$ to the unit
ball $\mathbb{B}^{n+1}$ which also satisfies the Neumann boundary condition:
  \begin{equation}\label{1.17}
  \begin{split}
\Delta^{k} v|_{\mathbb{S}^{n}}&=(-1)^{k}\frac{\Gamma(m+1)\Gamma(m-k+\frac{1}{2})}{\Gamma(m+\frac{1}{2})\Gamma(m-k+1)}\frac{\mathcal{P}_{2m+1}}{\mathcal{P}_{2m+1-2k}}f; \;0\leq k\leq[\frac{m}{2}];\\
\frac{\partial}{\partial n}\Delta^{k} v|_{\mathbb{S}^{n}}=&(-1)^{k+1}\frac{n-1-2m+2k}{2}\cdot\\
&\frac{\Gamma(m+1)\Gamma(m-k+\frac{1}{2})}{\Gamma(m+\frac{1}{2})\Gamma(m-k+1)}\frac{\mathcal{P}_{2m+1}}{\mathcal{P}_{2m+1-2k}}f,\;
\;\;\;\;\;\;\;0\leq k\leq[\frac{m-1}{2}].
\end{split}
\end{equation}
Then we have the inequality
  \begin{equation}\label{1.18}
  \begin{split}
\frac{\Gamma(m+1)\Gamma(\frac{1}{2})}{\Gamma(m+\frac{1}{2})}&\frac{\Gamma(\frac{n+2m+1}{2})}{\Gamma(\frac{n-2m-1}{2})}
\omega^{\frac{2m+1}{n}}_{n}\left(\int_{\mathbb{S}^{n}}|f|^{\frac{2n}{n-2m-1}}
d\sigma\right)^{\frac{n-2m-1}{n}}
\\ &\leq
\int_{\mathbb{B}^{n+1}}|\nabla^{m+1} v|^{2}dx+\int_{\mathbb{S}^{n}}f\mathcal{T}_{m}fd\sigma,
\end{split}
\end{equation}
where
  \begin{equation*}
  \begin{split}
\nabla^{m+1}=\left\{
               \begin{array}{ll}
                 \Delta^{\frac{m+1}{2}}, & \hbox{$m=$odd;} \\
                 \nabla\Delta^{\frac{m}{2}}, & \hbox{$m=$even,}
               \end{array}
             \right.
\end{split}
\end{equation*}
and $\mathcal{T}_{m}$ is an operator of order $2m$ defined as follow: if $m$ is an odd integer, then
\begin{equation*}
\begin{split}
\mathcal{T}_{m}=&\frac{n-1}{2}\frac{\Gamma(m+1)\Gamma(\frac{1}{2})}{\Gamma(m+\frac{1}{2})}
\frac{\mathcal{P}_{2m+1}}{\mathcal{P}_{1}}+\sum^{\frac{m-1}{2}}_{k=1}(m-2k)
\frac{\Gamma(m+1)^{2}}{\Gamma(m+\frac{1}{2})^{2}}\cdot\\
&\frac{\Gamma(k+\frac{1}{2})\Gamma(m-k+\frac{1}{2})}{\Gamma(k+1)\Gamma(m-k+1)}\frac{\mathcal{P}_{2m+1}^{2}}{\mathcal{P}_{2m+1-2k}\mathcal{P}_{2k+1}};
\end{split}
\end{equation*}
if $m$ is a even integer, then
\begin{equation*}
\begin{split}
\mathcal{T}_{m}=&\frac{n-1}{2}\frac{\Gamma(m+1)\Gamma(\frac{1}{2})}{\Gamma(m+\frac{1}{2})}
\frac{\mathcal{P}_{2m+1}}{\mathcal{P}_{1}}+
\frac{n-1-m}{2}\left(\frac{\Gamma(m+1)\Gamma(\frac{m+1}{2})}{\Gamma(m+\frac{1}{2})\Gamma(\frac{m}{2}+1)}\right)^{2}
\frac{\mathcal{P}_{2m+1}^{2}}{\mathcal{P}_{m+1}^{2}}\\
&+\sum^{\frac{m}{2}-1}_{k=1}(m-2k)
\frac{\Gamma(m+1)^{2}}{\Gamma(m+\frac{1}{2})^{2}}
\frac{\Gamma(k+\frac{1}{2})\Gamma(m-k+\frac{1}{2})}{\Gamma(k+1)\Gamma(m-k+1)}\frac{\mathcal{P}_{2m+1}^{2}}{\mathcal{P}_{2m+1-2k}\mathcal{P}_{2k+1}}.
\end{split}
\end{equation*}
Moreover, equality holds if and only if
\begin{equation}\label{b1.19}
  \begin{split}
v(x)=c\int_{\mathbb{S}^{n}}\frac{(1-|x|^{2})^{2m+1}}{|x-\xi|^{n+1+2m}}|1-\langle x_{0},\xi\rangle|^{\frac{2m+1-n}{4}}d\sigma,
\end{split}
\end{equation}
where $c$ is a constant and
$x_{0}$ is some point in $\mathbb{B}^{n+1}$. When $f=1$, inequality (\ref{1.18}) is
attained by the function $v(x)=\sum\limits^{m}_{k=0}\frac{(\frac{n-1}{2}-m)_{k}(-m)_{k}}{(-2m)_{k}}\frac{(2\rho)^{k}}{k!}$,
where $(a)_{k}$ is the rising Pochhammer symbol defined in Section 2.
\end{theorem}

\begin{theorem} \label{th1.7} Let $n\geq3$ be an odd integer.
Given $f\in C^{\infty}(\mathbb{S}^{n})$, suppose $v$ is a smooth extension of $f$ to the unit
ball $\mathbb{B}^{n+1}$ which also satisfies the Neumann boundary condition
  \begin{equation}\label{1.19}
  \begin{split}
\Delta^{k} v|_{\mathbb{S}^{n}}&=(-1)^{k}\frac{\Gamma(\frac{n+1}{2})\Gamma(\frac{n}{2}-k)}{\Gamma(\frac{n}{2})\Gamma(\frac{n+1}{2}-k)}
\frac{\mathcal{P}_{n}}{\mathcal{P}_{n-2k}}f; \;\;\;\;\;\;0\leq k\leq[\frac{n-1}{4}];\\
\frac{\partial}{\partial n}\Delta^{k} v|_{\mathbb{S}^{n}}&=(-1)^{k+1}k\frac{\Gamma(\frac{n+1}{2})\Gamma(\frac{n}{2}-k)}{\Gamma(\frac{n}{2})\Gamma(\frac{n+1}{2}-k)}
\frac{\mathcal{P}_{n}}{\mathcal{P}_{n-2k}}f,\;\;\;\;0\leq k\leq[\frac{n-3}{4}].
\end{split}
\end{equation}
Then we have the inequality
  \begin{equation}\label{1.20}
  \begin{split}
&\log \left(\frac{1}{\omega_{n}}\int_{\mathbb{S}^{n}}e^{n(f-\overline{f})}d\sigma\right)\\
\leq&
\frac{n}{2^{n+1}\pi^{\frac{n+1}{2}}\Gamma(\frac{n+1}{2})}
\left(\int_{\mathbb{B}^{n+1}}|\nabla^{\frac{n+1}{2}} v|^{2}dx+\int_{\mathbb{S}^{n}}f\mathcal{T}_{\frac{n-1}{2}}fd\sigma\right),
\end{split}
\end{equation}
where the operator $\mathcal{T}_{\frac{n-1}{2}}$ is defined in Theorem \ref{th1.6}.
Moreover, equality holds if and only if
  \begin{equation}\label{b1.20}
  \begin{split}
v(x)=\pi^{-\frac{n}{2}}
\frac{\Gamma(n)}{2^{n}\Gamma(\frac{n}{2})}\int_{\mathbb{S}^{n}}\frac{(1-|x|^{2})^{n}}{|x-\xi|^{2n}}(-\ln|1-\langle x_{0},\xi\rangle|+c)d\sigma
\end{split}
\end{equation}
where $c$ is a constant and
$x_{0}$ is some point in $\mathbb{B}^{n+1}$.
\end{theorem}

\begin{remark}
We remark that one can also replace the  Neumann boundary condition (\ref{1.17}) or (\ref{1.19}) by
$\partial_{n}v|_{\mathbb{S}^{n}}=\partial_{n}V_{m}|_{\mathbb{S}^{n}}, \partial^{2}_{n}v|_{\mathbb{S}^{n}}=\partial^{2}_{n}V_{m}|_{\mathbb{S}^{n}} ,\cdots$.
When $k$ is small, we can compute the value through  (\ref{4.2}) and (\ref{4.3}). For example, we have
\[
\partial_{n}V_{m}|_{\mathbb{S}^{n}}=\frac{2m+1-n}{2}f,\;\partial^{2}_{n}V_{m}|_{\mathbb{S}^{n}}=\frac{\widetilde{\Delta}f}{2m-1}+
\frac{(n-1-2m)[(m-1)n-m(2m-1)]}{2(2m-1)}f.
\]
However, the argument  would become increasingly delicate when $k$ is large.
\end{remark}

\begin{theorem} \label{th1.8} Let $n>2m+1\geq3$ and
 $U(x,y)\in W^{m+1,2}(\overline{\mathbb{R}_{+}^{n+1}})$ be
satisfied the Neumann boundary condition
  \begin{equation}\label{1.21}
  \begin{split}
\Delta^{k}U_{m}(x,y)|_{y=0}=&\frac{\Gamma(m+1)\Gamma(m+\frac{1}{2}-k)}{\Gamma(m-k-1)\Gamma(m+\frac{1}{2})}\Delta^{k}_{x}f,
\;\;\;0\leq k\leq [\frac{m}{2}];\\
\partial_{y}\Delta^{k}U_{m}(x,y)|_{y=0}=&0,\;\;\;\;\;\;\;\;\;\;\;\;\;\;\;\;\;\;\;\;\;\;\;\;\;\;\;\;\;
\;\;\;\;\;\;\;\;0\leq k\leq [\frac{m-1}{2}].
\end{split}
\end{equation}
Then we have the sharp trace inequality
  \begin{equation*}
  \begin{split}
\frac{\Gamma(m+1)\Gamma(\frac{1}{2})}{\Gamma(m+\frac{1}{2})}&\frac{\Gamma(\frac{n+2m+1}{2})}{\Gamma(\frac{n-2m-1}{2})}
\omega^{\frac{2m+1}{n}}_{n}\left(\int_{\mathbb{R}^{n}}|U(x,0)|^{\frac{2n}{n-2m-1}}dx\right)^{\frac{n-2m-1}{n}}
\leq
\int_{\mathbb{R}_{+}^{n+1}}|\nabla^{m+1} U|^{2}dxdy.
\end{split}
\end{equation*}
Furthermore, equality  holds if and only
  \begin{equation}\label{b1.23}
  \begin{split}
U(x,y)=c\int_{\mathbb{R}^{n}}\frac{y^{1+2m}}{(|x-\xi|^{2}+y^{2})^{\frac{n+1}{2}+m}}(\lambda^{2}+|\xi-\xi_{0}|^{2})^{-\frac{n-2m-1}{2}}d\xi,
\end{split}
\end{equation}
where $\lambda>0$,  $c$ is a constant and
$\xi_{0}$ is some fixed point in $\mathbb{R}^{n}$.
\end{theorem}
\begin{remark}
We remark that, because of Lemma \ref{lm5.4}, one can also replace the Neumann boundary condition (\ref{1.21}) by the
following:
  \begin{equation}\label{1.22}
  \begin{split}
\partial^{2k}_{y}U(x,y)|_{y=0}=&\frac{\Gamma(k+\frac{1}{2})\Gamma(m-k+\frac{1}{2})}{\Gamma(\frac{1}{2})\Gamma(m+\frac{1}{2})}\Delta_{x}^{k}f,
\;0\leq k\leq [\frac{m}{4}];\\
\partial^{2k+1}_{y}U(x,y)|_{y=0}=&0,\;\;\;\;\;\;\;\;\;\;\;\;\;\;\;\;\;\;\;\;\;\;\;\;\;\;
\;\;\;\;\;\;\;\;0\leq k\leq [\frac{m-1}{4}].
\end{split}
\end{equation}
Notice that the boundary condition (\ref{1.22}) is different to that given by R. Yang (\cite{y}).
\end{remark}

This article is organized as follows:
In Section 2, we briefly quote  some of properties of special functions, such as
hypergeometric function and Gegenbauer polynomials, and
Funk-Hecke formula for spherical harmonics which will be used in the paper.
In Section 3 we first review the connection between scattering theory  and conformally invariant objects on their boundaries. Next we
  compute the explicit formulas of the solution  of Poisson equation  and the adapted metrics on the model case $(\mathbb{B}^{n+1},\mathbb{S}^{n},g_{\mathbb{B}})$.
In Section 4, we prove Theorem \ref{th1.6} and \ref{th1.7}. The proof of Theorem \ref{th1.8}
 is given in Section 5.

\section{Preliminaries}
In this section, we  quote  some preliminary facts which will be needed in
the sequel.
\subsection{Hypergeometric function}
We use the notation $F(a,b;c;z)$ to denote
  \begin{equation}\label{2.1}
  \begin{split}
F(a,b;c;z)=\sum^{\infty}_{k=0}\frac{(a)_{k}(b)_{k}}{(c)_{k}}\frac{z^{k}}{k!},
\end{split}
\end{equation}
where $c\neq0,-1,\cdots,-n,\cdots$ and $(a)_{k}$ is the rising Pochhammer symbol defined by
$$
(a)_{0}=0,\;(a)_{k}=a(a+1)\cdots(a+k-1), \;k\geq1.
$$
If either $a$ or $b$ is a nonpositive integer, then the series terminates and  the function reduces to a polynomial.

 Here, we
only list some of properties of hypergeometric function which will be used in the rest of paper. For
more information about of these functions, we refer to \cite{g}, section 9.1 and \cite{er}, Chapter II.
\begin{itemize}
  \item The hypergeometric function $F(a,b;c;z)$ satisfies  the hypergeometric differential equation
  \begin{equation}\label{2.2}
  \begin{split}
z(1-z)F''+(c-(a+b+1)z)F'-abF=0.
\end{split}
\end{equation}

  \item If $\textrm{Re} (c-a-b)>0$, then $F(a,b;c;1)$ exists and
      \begin{equation}\label{2.3}
  \begin{split}
F(a,b;c;1)=\frac{\Gamma(c)\Gamma(c-a-b)}{\Gamma(c-a)\Gamma(c-b)}.
\end{split}
\end{equation}

  \item Transformation formulas (1):
    \begin{equation}\label{2.4}
  \begin{split}
F(a,b;c;z)=(1-z)^{c-a-b} F(c-a,c-b;c;z).
\end{split}
\end{equation}

 \item Transformation formulas (2): if $c-a-b$ is not an integer, then
    \begin{equation}\label{2.5}
  \begin{split}
F(a,b;c;z)=&\frac{\Gamma(c)\Gamma(c-a-b)}{\Gamma(c-a)\Gamma(c-b)} F(a,b;a+b-c+1;1-z)+\\
&(1-z)^{c-a-b}\frac{\Gamma(c)\Gamma(a+b-c)}{\Gamma(a)\Gamma(b)}F(c-a,c-b;c-a-b+1;1-z).
\end{split}
\end{equation}

 \item Differentiation formula:
    \begin{equation}\label{2.6}
  \begin{split}
\frac{d^{k}}{dz^{k}}F(a,b;c;z)=\frac{(a)_{k}(b)_{k}}{(c)_{k}}F(a+k,b+k;c+k;z),\;\;k\geq1.
\end{split}
\end{equation}

\end{itemize}

\subsection{Gegenbauer polynomials}
We use the notation $C^{\alpha}_{k}(x)$ to denote the Gegenbauer polynomial of degree $k$ which
 can be defined in terms of the generating function:
    \begin{equation}\label{2.7}
  \begin{split}
\frac{1}{(1-2xt+t^{2})^{\alpha}}=\sum^{\infty}_{k=0}C^{\alpha}_{k}(x)t^{k}.
\end{split}
\end{equation}

Here, we also
list some of properties of Gegenbauer polynomial  and  refer to \cite{sw} and \cite{g}, section 8.93
for more information about this polynomial.
\begin{itemize}
  \item Rodrigues formula:
    \begin{equation}\label{2.8}
  \begin{split}
C^{\alpha}_{k}(x)=\frac{(-1)^{k}}{2^{k}k!}\frac{\Gamma(\alpha+\frac{1}{2})\Gamma(k+2\alpha)}{\Gamma(2\alpha)\Gamma(\alpha+k+\frac{1}{2})}
(1-x^{2})^{-\alpha+\frac{1}{2}}\frac{d^{k}}{dx^{k}}(1-x^{2})^{k+\alpha-\frac{1}{2}}.
\end{split}
\end{equation}

  \item Orthogonality and normalization: if $k\neq m$, then
      \begin{equation}\label{2.9}
  \begin{split}
\int^{1}_{-1}C^{\alpha}_{k}(x)C^{\alpha}_{m}(x)(1-x^{2})^{\alpha-\frac{1}{2}}dx=0;
\end{split}
\end{equation}
if $k=m$, then
      \begin{equation}\label{2.10}
  \begin{split}
\int^{1}_{-1}[C^{\alpha}_{k}(x)]^{2}(1-x^{2})^{\alpha-\frac{1}{2}}dx=\frac{\pi2^{1-2\alpha}\Gamma(k+2\alpha)}{k!(k+\alpha)[\Gamma(\alpha)]^{2}}.
\end{split}
\end{equation}

  \item Differentiation formulas:
        \begin{equation}\label{2.11}
  \begin{split}
\frac{d^{m}}{dx^{m}}C^{\alpha}_{k}(x)=\left\{
                                        \begin{array}{ll}
                                          2^{m}\frac{\Gamma(\alpha+m)}{\Gamma(\alpha)}C^{\alpha+m}_{k-m}(x), & \hbox{$k-m\geq0$;} \\
                                          0, & \hbox{$k-m<0$.}
                                        \end{array}
                                      \right.
\end{split}
\end{equation}

\end{itemize}

Finally, we recall an integral (see \cite{g}, page 407, 3.665)
$$
\int^{\pi}_{0}\frac{\sin^{2\mu-1}\theta}{(1-2t\cos\theta+t^{2})^{\alpha}}d\theta=\frac{\Gamma(\mu)\Gamma(\frac{1}{2})}{\Gamma(\mu+\frac{1}{2})}
F(\alpha,\alpha-\mu+\frac{1}{2};\mu+\frac{1}{2};t^{2}),\;\; \textrm{Re} \mu>0, |t|<1.
$$
Using the expansion (\ref{2.1}) and (\ref{2.7}), we have, for $k\geq0$,
      \begin{equation}\label{2.12}
  \begin{split}
\int^{\pi}_{0}C^{\alpha}_{2k}(\cos\theta)\sin^{2\mu-1}\theta dx=&\frac{\Gamma(\mu)\Gamma(\frac{1}{2})}{\Gamma(\mu+\frac{1}{2})}
\frac{(\alpha)_{k}(\alpha-\mu+\frac{1}{2})_{k}}{(\mu+\frac{1}{2})_{k}}\frac{1}{k!};\\
\int^{\pi}_{0}C^{\alpha}_{2k+1}(\cos\theta)\sin^{2\mu-1}\theta dx=&0.
\end{split}
\end{equation}

\subsection{Funk-Hecke formula} It is known that $L^{2}(\mathbb{S}^{n})$ can be decomposed as follow
\[
L^{2}(\mathbb{S}^{n})=\bigoplus^{\infty}_{l=0}\mathcal{H}_{l},
\]
where $\mathcal{H}_{l}$ is the space of spherical harmonics of degree $l$ (see \cite{sw}).
 For $n\geq2$, the
Funk-Hecke formula reads as follow (see e.g. \cite{be,fr})
\begin{equation}\label{2.13}
  \begin{split}
\int_{\mathbb{S}^{n}}K(\langle\xi,\eta\rangle)Y_{l}(\eta)d\sigma(\eta)=&\lambda_{l}Y_{l},\\
\lambda_{l}=&(4\pi)^{\frac{n-1}{2}}\frac{l!\Gamma(\frac{n-1}{2})}{\Gamma(l+n-1)}\int^{1}_{-1}K(t)C^{\frac{n-1}{2}}_{l}(t)(1-t^{2})^{\frac{n-2}{2}}dt,
\end{split}
\end{equation}
where $K\in L^{1}((-1,1), (1-t^{2})^{\frac{n-2}{2}}dt)$ and $Y_{l}\in\mathcal{H}_{l}$.
Moreover, if $Y_{l}\in\mathcal{H}_{l}$, then
\begin{equation}\label{2.14}
  \begin{split}
-\widetilde{\Delta}Y_{l}=l(n-1+l)Y_{l}
\end{split}
\end{equation}
and thus
\begin{equation}\label{2.15}
  \begin{split}
BY_{l}=\left(l+\frac{n-1}{2}\right)Y_{l},\;\;\mathcal{P}_{2\gamma}Y_{l}=\frac{\Gamma(l+\frac{n}{2}+\gamma)}{\Gamma(l+\frac{n}{2}-\gamma)}Y_{l},
\end{split}
\end{equation}
where $B$ and $\mathcal{P}_{2\gamma}$ defined in (\ref{1.7}).

\section{adapted metrics}
Firstly, we briefly review  the definition of the fractional GJMS operator via scattering theory (see \cite{grz}).
A triple $(X^{n+1}, M^{n}, g_{+})$ is a Poincar\'e-Einstein manifold if

(1) $X^{n+1}$ is (diffeomorphic to) the interior of a compact manifold $\overline{X}^{n+1}$ with boundary $\partial\overline{X}=M^{n}$,

(2) $X^{n+1}$ is complete with $\textrm{Ric}(g_{+})=-ng_{+}$, and

(3) there exists a nonnegative $\rho\in C^{\infty}(X)$ such that $\rho^{-1}(0)=M^{n}$,
$d\rho\neq0$ along $M$, and the metric $g:=\rho^{2}g_{+}$ extends to a smooth metric on
$\overline{X}^{n+1}$.

A  function $\rho$ satisfying (3) above is called a \emph{defining function}. It is obvious that
the conformal class $[h]:=[g|_{TM}]$ on $M$ is well-defined for a Poincar\'e-Einstein
manifold because $\rho$ is only
determined up to multiplication by a positive smooth function on $X$.

Given a Poincar\'e-Einstein manifold $(X^{n+1}, M^{n}, g_{+})$ and a representative $[h]$
on the conformal boundary, there is a uniquely defining function $\rho$ such that
$g_{+}=\rho^{-2}(d\rho^{2}+h_{\rho})$ on $M\times(0,\delta)$, where $h_{\rho}$
is a one-parameter
family of metrics on $M$ satisfying $h_{0}=h$.
Given $f\in C^{\infty}(M)$. It has been shown (see \cite{ma,grz}) that the Poisson equation
\begin{equation}\label{3.1}
\begin{split}
-\Delta_{g_{+}}u-s(n-s)u=0
\end{split}
\end{equation}
has a unique solution of the form
\begin{equation}\label{3.2}
\begin{split}
u=F\rho^{n-s} +H\rho^{s},\;\; F,H\in C^{\infty}(X),\; F|_{\rho=0}=f,
\end{split}
\end{equation}
where $s\in\mathbb{C}$ and $s(n-s)$ do not belongs to the pure point spectrum of $-\Delta_{g_{+}}$.
The scattering
operator on $M$ is defined as $S(s)f=H|_{M}$. If $\textrm{Re}(s)>\frac{n}{2}$, then the scattering
operator is a meromorphic family of pseudo-differential
operators. Graham and Zworski \cite{grz} defined the fractional GJMS operator $P_{2\gamma} (\gamma\in(0,\frac{n}{2})\setminus\mathbb{N})$ as follow
\begin{equation}\label{3.3}
\begin{split}
P_{2\gamma}f:=d_{\gamma}S\left(\frac{n}{2}+\gamma\right)f,\;\;d_{\gamma}=2^{2\gamma}\frac{\Gamma(\gamma)}{\Gamma(-\gamma)}.
\end{split}
\end{equation}
Here we denote by $\mathbb{N}$  the set of all natural numbers and $\mathbb{N}_{0}=\mathbb{N}\setminus\{0\}$.
In term of $P_{2\gamma}$, the the fractional $Q$-curvature $Q_{2\gamma}$ is defined by
\begin{equation*}
\begin{split}
Q_{2\gamma}:=\frac{2}{n-2\gamma}P_{2\gamma}(1).
\end{split}
\end{equation*}
 If $\gamma\in\mathbb{N}_{0}$, then $P_{2\gamma}$ is nothing but the GJMS operator on $M$ (see \cite{gr}).
It has been also shown by Graham and Zworski \cite{grz}  that the
principal symbol of $P_{2\gamma}$ is is exactly the principal
symbol of the fractional Laplacian $(-\Delta)^{\gamma}$ and
satisfy an important conformal covariance property:
for a conformal change of metric $\widehat{h}=e^{2\tau}h$,
we have
\begin{equation}\label{3.4}
\begin{split}
\widehat{P}_{2\gamma}f=e^{-\frac{n+2\gamma}{2}\tau}P_{2\gamma}\left(e^{\frac{n-2\gamma}{2}}f\right),\;\;\forall f\in C^{\infty}(M).
\end{split}
\end{equation}

Next we recall the adapted metric,  introduced by Case and Chang \cite{cac}, on a conformally compact Poincar\'e-Einstein manifold
$(X^{n+1}, \partial X, g_{+})$. This metric is introduced for any parameter
$s=\frac{n}{2}+\gamma$ with $\gamma\in(0,\frac{n}{2})$  and $s=n$ if $n$ is odd. For such an $s$, we denote by $\vartheta_{s}$
the solution of Poisson equation (\ref{3.1}) with Dirichlet condition $f\equiv1$.
Notice that if   the Yamabe constant of the
boundary metric $h$ is positive, then by a result of Lee (see \cite{le}, Theorem A), we have $\vartheta_{s}>0$ so that one can
take $\rho_{s}:=(\vartheta_{s})^{\frac{1}{n-s}}$ as a defining function. The metric $g_{s}=\rho^{2}_{s}g_{+}$ is called
adapted metric. In the limiting case, namely $s=n$ and $n$ is an odd integer, the  adapted metric, appeared in \cite{fe1},
is defined as $g^{\ast}=e^{2\tau}$, where
\begin{equation}\label{3.5}
\begin{split}
\tau=-\frac{d}{ds}\vartheta_{s}|_{s=n}.
\end{split}
\end{equation}
 We remark that $\tau$ satisfies $-\Delta_{g_{+}}\tau=n$.
For more information about  GJMS operator and adapted metric, we refer to \cite{c,ca1,ca2,cac,cal,chg,chw,chy,fe1,fe3,grl,grz,ju,le,y}

In the rest of this section, we shall consider the model case $(\mathbb{B}^{n+1},\mathbb{S}^{n},g_{\mathbb{B}})$, where $g_{\mathbb{B}}=\frac{4}{(1-|x|^{2})^{2}}g_{0}$
and $g_{0}=|dx|^{2}$ is the
Euclidean metric. The defining function is $\rho=\frac{1-|x|^{2}}{2}$.
Firstly, we give the explicit formula of the solution of Poisson equation (\ref{3.1}).
The main result is the following theorem:
\begin{theorem}\label{th3.1}
Let $\gamma\in (0,\frac{n}{2})$, $s=\frac{n}{2}+\gamma$ and $\rho=\frac{1-|x|^{2}}{2}$. The solution of the following
 Poisson equation on the hyperbolic space $(\mathbb{B}^{n+1},g_{\mathbb{B}})$
\begin{equation}\label{a3.1}
\begin{split}
\left\{
  \begin{array}{ll}
    -\Delta_{g_{\mathbb{B}}}u-s(n-s)u=0 & \hbox{in $\mathbb{B}^{n+1}$,} \\
    u=F\rho^{n-s}+H\rho^{s}, & \hbox{} \\
 F|_{\partial \mathbb{B}^{n+1}}=f(\xi), & \hbox{.}
  \end{array}
\right.
\end{split}
\end{equation}
is
\begin{equation}\label{3.7}
\begin{split}
u(x)=\pi^{-\frac{n}{2}}
\frac{\Gamma(\frac{n}{2}+\gamma)}{\Gamma(\gamma)}
\int_{\mathbb{S}^{n}}\left(\frac{1-|x|^{2}}{2|x-\xi|^{2}}\right)^{s}f(\xi)d\sigma.
\end{split}
\end{equation}
Furthermore, if $f$ has an expansion in spherical harmonics, $f=\sum\limits^{\infty}_{l=0} Y_{l}$, where $Y_{l}$ is a
spherical harmonic of degree $l$, then (here we set $r=|x|$)
\begin{equation}\label{3.8}
\begin{split}
u(x)=\rho^{n-s}
\sum^{\infty}_{l=0}\varphi_{l}(r^{2})r^{l}Y_{l},
\end{split}
\end{equation}
where
\begin{equation}\label{3.9}
\begin{split}
\varphi_{l}(r)
=&\frac{\Gamma(\gamma+\frac{1}{2})}{\Gamma(2\gamma)}
\frac{\Gamma(l+\frac{n}{2}+\gamma)}{\Gamma(l+\frac{n+1}{2})}
F(l+\frac{n}{2}-\gamma,\frac{1}{2}-\gamma,l+\frac{n+1}{2};r)
\end{split}
\end{equation}
satisfying $\varphi_{l}(1)=1.$
\end{theorem}
\begin{proof}
Set
\begin{equation*}
\begin{split}
V(x)=&\pi^{-\frac{n}{2}}
\frac{\Gamma(\frac{n}{2}+\gamma)}{\Gamma(\gamma)}
\rho^{2\gamma}\int_{\mathbb{S}^{n}}\left(\frac{1}{|x-\xi|^{2}}\right)^{s}f(\xi)d\sigma\\
=&\pi^{-\frac{n}{2}}
\frac{\Gamma(\frac{n}{2}+\gamma)}{\Gamma(\gamma)}
\rho^{2\gamma}\int_{\mathbb{S}^{n}}\frac{f(\xi)}{(1-2x\cdot\xi+|x|^{2})^{\frac{n}{2}+\gamma}}d\sigma.
\end{split}
\end{equation*}
 By Funk-Hecke formula (\ref{2.13}), if
$f=\sum\limits^{\infty}_{l=0} Y_{l}$, then
\begin{equation}\label{3.10}
\begin{split}
V(x)=&\pi^{-\frac{n}{2}}
\frac{\Gamma(\frac{n}{2}+\gamma)}{\Gamma(\gamma)}\rho^{2\gamma}\sum^{\infty}_{l=0}\lambda_{l}Y_{l},
\end{split}
\end{equation}
where
\begin{equation}\label{3.11}
\begin{split}
\lambda_{l}=&(4\pi)^{\frac{n-1}{2}}\frac{l!\Gamma(\frac{n-1}{2})}{\Gamma(l+n-1)}\int^{1}_{-1}
\frac{1}{(1-2rt+r^{2})^{\frac{n}{2}+\gamma}}C^{\frac{n-1}{2}}_{l}(t)(1-t^{2})^{\frac{n-2}{2}}dt\\
=&(4\pi)^{\frac{n-1}{2}}\frac{l!\Gamma(\frac{n-1}{2})}{\Gamma(l+n-1)}\sum^{\infty}_{k=0}r^{k}\int^{1}_{-1}
C^{\frac{n}{2}+\gamma}_{k}(t)C^{\frac{n-1}{2}}_{l}(t)(1-t^{2})^{\frac{n-2}{2}}dt.
\end{split}
\end{equation}
Using the Rodrigues formula (\ref{2.8}) and differentiation formula (\ref{2.11}), we have
\begin{equation}\label{3.12}
\begin{split}
&\sum^{\infty}_{k=0}r^{k}\int^{1}_{-1}
C^{\frac{n}{2}+\gamma}_{k}(t)C^{\frac{n-1}{2}}_{l}(t)(1-t^{2})^{\frac{n-2}{2}}dt\\
=&\frac{(-1)^{l}}{2^{l}l!}\frac{\Gamma(\frac{n}{2})\Gamma(l+n-1)}{\Gamma(n-1)\Gamma(l+\frac{n}{2})}\sum^{\infty}_{k=0}r^{k}
\int^{1}_{-1}
C^{\frac{n}{2}+\gamma}_{k}(t)\frac{d^{l}}{dt^{l}}(1-t^{2})^{l+\frac{n-2}{2}}dt\\
=&\frac{1}{2^{l}l!}\frac{\Gamma(\frac{n}{2})\Gamma(l+n-1)}{\Gamma(n-1)\Gamma(l+\frac{n}{2})}\sum^{\infty}_{k=l}2^{l}
\frac{\Gamma(\frac{n}{2}+\gamma+l)}{\Gamma(\frac{n}{2}+\gamma)}r^{k}
\int^{1}_{-1}
C^{\frac{n}{2}+\gamma+l}_{k-l}(t)(1-t^{2})^{l+\frac{n-2}{2}}dt\\
=&\frac{1}{l!}\frac{\Gamma(\frac{n}{2})\Gamma(l+n-1)}{\Gamma(n-1)\Gamma(l+\frac{n}{2})}
\frac{\Gamma(\frac{n}{2}+\gamma+l)}{\Gamma(\frac{n}{2}+\gamma)}\sum^{\infty}_{k=0}
r^{l+k}
\int^{1}_{-1}
C^{\frac{n}{2}+\gamma+l}_{k}(t)(1-t^{2})^{l+\frac{n-2}{2}}dt.
\end{split}
\end{equation}
Substituting $t=\cos\theta$, we have, by (\ref{2.12}),
\begin{equation}\label{3.13}
\begin{split}
&\sum^{\infty}_{k=0}
r^{l+k}
\int^{1}_{-1}
C^{\frac{n}{2}+\gamma+l}_{0}(t)(1-t^{2})^{l+\frac{n-2}{2}}dt
=\sum^{\infty}_{k=0}
r^{l+k}
\int^{\pi}_{0}
C^{\frac{n}{2}+\gamma+l}_{k}(\cos\theta)\sin^{2l+n-1}\theta d\theta\\
=&\sum^{\infty}_{k=0}
r^{l+2k}
\int^{\pi}_{0}
C^{\frac{n}{2}+\gamma+l}_{2k}(\cos\theta)\sin^{2l+n-1}\theta d\theta
=\frac{\Gamma(l+\frac{n}{2})\Gamma(\frac{1}{2})}{\Gamma(l+\frac{n+1}{2})}\sum^{\infty}_{k=0}
r^{l+2k}
\frac{(\frac{n}{2}+\gamma+l)_{k}(\gamma+\frac{1}{2})_{k}}{(l+\frac{n+1}{2})_{k}}\frac{1}{k!}.
\end{split}
\end{equation}
Combing (\ref{3.11})-(\ref{3.13}) and (\ref{2.4}) yields
\begin{equation}\label{3.14}
\begin{split}
\lambda_{l}=&(4\pi)^{\frac{n-1}{2}}\frac{l!\Gamma(\frac{n-1}{2})}{\Gamma(l+n-1)}
\frac{1}{l!}\frac{\Gamma(\frac{n}{2})\Gamma(l+n-1)}{\Gamma(n-1)\Gamma(l+\frac{n}{2})}
\frac{\Gamma(\frac{n}{2}+\gamma+l)}{\Gamma(\frac{n}{2}+\gamma)}\frac{\Gamma(l+\frac{n}{2})\Gamma(\frac{1}{2})}{\Gamma(l+\frac{n+1}{2})}\cdot\\
&\sum^{\infty}_{k=0}
r^{l+2k}
\frac{(\frac{n}{2}+\gamma+l)_{k}(\gamma+\frac{1}{2})_{k}}{(l+\frac{n+1}{2})_{k}}\frac{1}{k!}\\
=&(4\pi)^{\frac{n-1}{2}}\frac{\Gamma(\frac{n-1}{2})\Gamma(\frac{n}{2})\Gamma(\frac{1}{2})}{\Gamma(n-1)}
\frac{\Gamma(\frac{n}{2}+\gamma+l)}{\Gamma(\frac{n}{2}+\gamma)\Gamma(\frac{n+1}{2}+l)}
F(\gamma+\frac{1}{2},\frac{n}{2}+l+\gamma,l+\frac{n+1}{2};r^{2})r^{l}\\
=&(4\pi)^{\frac{n-1}{2}}\frac{\Gamma(\frac{n-1}{2})\Gamma(\frac{n}{2})\Gamma(\frac{1}{2})}{\Gamma(n-1)}
\frac{\Gamma(\frac{n}{2}+\gamma+l)}{\Gamma(\frac{n}{2}+\gamma)\Gamma(\frac{n+1}{2}+l)}\cdot\\
&(1-r^{2})^{-2\gamma}F(l+\frac{n}{2}-\gamma,\frac{1}{2}-\gamma,l+\frac{n+1}{2};r^{2})r^{l}\\
=&2^{-2\gamma}(4\pi)^{\frac{n-1}{2}}\frac{\Gamma(\frac{n-1}{2})\Gamma(\frac{n}{2})\Gamma(\frac{1}{2})}{\Gamma(n-1)}
\frac{\Gamma(2\gamma)}{\Gamma(\gamma+\frac{1}{2})\Gamma(\frac{n}{2}+\gamma)}\rho^{-2\gamma}\varphi_{l}(r^{2})r^{l}.
\end{split}
\end{equation}
Therefore,
\begin{equation}\label{3.15}
\begin{split}
V(x)=&\pi^{-\frac{n}{2}}
\frac{\Gamma(\frac{n}{2}+\gamma)}{\Gamma(\gamma)}\rho^{2\gamma}\sum^{\infty}_{l=0}\lambda_{l}Y_{l}\\
=&2^{n-1-2\gamma}\frac{\Gamma(2\gamma)}{\Gamma(\gamma)\Gamma(\gamma+\frac{1}{2})}
\frac{\Gamma(\frac{n-1}{2})\Gamma(\frac{n}{2})}{\Gamma(n-1)}\rho^{2\gamma}\rho^{-2\gamma}\sum^{\infty}_{l=0}\varphi_{l}(r^{2})r^{l}Y_{l}\\
=&\sum^{\infty}_{l=0}\varphi_{l}(r^{2})r^{l}Y_{l}.
\end{split}
\end{equation}
To get the last equation above, we use the   duplication formula
\begin{equation}\label{3.16}
\begin{split}
\Gamma(2z)=2^{2z-1}\frac{\Gamma(z)\Gamma(z+\frac{1}{2})}{\Gamma(\frac{1}{2})}.
\end{split}
\end{equation}
By (\ref{2.3}), we have $\varphi_{l}(1)=1$ and thus $V(x)|_{r=1}=f(x)$.

By the uniqueness of the solution, to finish the proof, it is enough to show
\begin{equation}\label{3.17}
\begin{split}
-\Delta_{g_{\mathbb{B}}} \left[\rho^{s-2\gamma}V(x)\right]-s(n-s)\rho^{s-2\gamma}V(x)=0,
\end{split}
\end{equation}
or, equivalently,
\begin{equation}\label{3.18}
\begin{split}
-\Delta_{g_{\mathbb{B}}} \left[\rho^{n-s}\varphi_{l}(r^{2})r^{l}Y_{l}\right]-s(n-s)\rho^{n-s}\varphi_{l}(r^{2})r^{l}Y_{l}=0,\;\;\forall l\geq0.
\end{split}
\end{equation}

Recall the conformal laplacian on $(\mathbb{B}^{n+1},g_{\mathbb{B}})$ is
$$L_{g\mathbb{B}}=-\Delta_{g_{\mathbb{B}}}+\frac{n-1}{4n}\textrm{Scal}(g_{\mathbb{B}}),$$
where $\textrm{Scal}(g_{\mathbb{B}})$
ia scalar curvature on $(\mathbb{B}^{n+1},g_{\mathbb{B}})$. Since $(\mathbb{B}^{n+1},g_{\mathbb{B}})$  has the
constant  sectional curvature $-1$, we have $\textrm{Scal}(g_{\mathbb{B}})=-n(n+1)$ and thus
\begin{equation}\label{3.19}
\begin{split}
L_{g_{\mathbb{B}}}=-\Delta_{g_{\mathbb{B}}}-\frac{n^{2}-1}{4}.
\end{split}
\end{equation}
By
the conformal covariant property of the conformal Laplacian
for the change of metric, we have
\begin{equation}\label{3.20}
\begin{split}
L_{g_{\mathbb{B}}}f=\rho^{\frac{n+3}{2}}(-\Delta)\left(\rho^{-\frac{n-1}{2}}f\right), \;\;\forall f\in C^{\infty}(\mathbb{B}^{n+1}),
\end{split}
\end{equation}
where $\Delta$ is the Laplacian on Euclidean space.
We have, by (\ref{3.19}) and (\ref{3.20}),
\begin{equation}\label{3.21}
\begin{split}
&\left(\Delta_{g_{\mathbb{B}}}+\frac{n^{2}-1}{4}\right)
\rho^{n-s}\varphi_{l}(r^{2})r^{l}Y_{l}
=\rho^{\frac{n+3}{2}}\Delta\left(\rho^{-\frac{n-1}{2}}\rho^{n-s}\varphi_{l}(r^{2})r^{l}Y_{l}\right)\\
=&\rho^{\frac{n+3}{2}}\Delta\left(\rho^{\frac{1}{2}-\gamma}\varphi_{l}(r^{2})r^{l}Y_{l}\right)\\
=&\rho^{\frac{n+3}{2}}\left[\Delta\left(\rho^{\frac{1}{2}-\gamma}\varphi_{l}(r^{2})\right)r^{l}Y_{l}+
2\left\langle\nabla\left(\rho^{\frac{1}{2}-\gamma}\varphi_{l}(r^{2})\right), \nabla r^{l}Y_{l}\right\rangle
\right].
\end{split}
\end{equation}
To get the last equation, we use the fact $\Delta r^{l}Y_{l}=0$
since $Y_{l}$ is the spherical harmonic of degree $l$.
Substituting the polar  coordinate formula
$\Delta=\frac{\partial^{2}}{\partial r^{2}}+\frac{n}{r}\frac{\partial}{\partial r}+\frac{1}{r^{2}}\widetilde{\Delta}$
into (\ref{3.21}), we have
\begin{equation}\label{3.22}
\begin{split}
&\left(\Delta_{g_{\mathbb{B}}}+\frac{n^{2}-1}{4}\right)
\rho^{n-s}\varphi_{l}(r^{2})r^{l}Y_{l}\\
=&\rho^{\frac{n+3}{2}}
\left[\left(\rho^{\frac{1}{2}-\gamma}\varphi_{l}(r^{2})\right)''+\frac{n+2l}{r}
\left(\rho^{\frac{1}{2}-\gamma}\varphi_{l}(r^{2})\right)'\right]r^{l}Y_{l}.
\end{split}
\end{equation}
We compute
\begin{equation}\label{3.23}
\begin{split}
\left(\rho^{\frac{1}{2}-\gamma}\varphi_{l}(r^{2})\right)'=&2r\rho^{\frac{1}{2}-\gamma}\varphi'_{l}(r^{2})+(\gamma-\frac{1}{2})r
\rho^{-\frac{1}{2}-\gamma}\varphi_{l}(r^{2});\\
\left(\rho^{\frac{1}{2}-\gamma}\varphi_{l}(r^{2})\right)''=&4r^{2}\rho^{\frac{1}{2}-\gamma}
\varphi''_{l}(r^{2})+2\rho^{-\frac{1}{2}-\gamma}\left[\rho-(1-2\gamma)r^{2}\right]\varphi'_{l}(r^{2})+\\
&\left[(\gamma-\frac{1}{2})\rho^{-\frac{1}{2}-\gamma}+(\gamma^{2}-\frac{1}{4})r^{2}\rho^{-\frac{3}{2}-\gamma}\right]\varphi_{l}(r^{2}).
\end{split}
\end{equation}
Substituting (\ref{3.23}) into (\ref{3.22}), we obtain
\begin{equation}\label{3.24}
\begin{split}
&\left(\Delta_{g_{\mathbb{B}}}+\frac{n^{2}-1}{4}\right)
\rho^{n-s}\varphi_{l}(r^{2})r^{l}Y_{l}\\
=&\rho^{\frac{n}{2}-\gamma}\left\{4r^{2}\rho^{2}
\varphi''_{l}(r^{2})+2\rho\left[l+\frac{n+1}{2}-\left(l+\frac{n+3}{2}-2\gamma\right)r^{2}\right]\varphi'_{l}(r^{2})\right.\\
&\left. +\left[(n+2l+1)(\gamma-\frac{1}{2})\rho+(\gamma^{2}-\frac{1}{4})r^{2}\right]\varphi_{l}(r^{2})\right\}r^{l}Y_{l}.
\end{split}
\end{equation}
By (\ref{2.2}),  $\varphi_{l}(r^{2})$ satisfies  the hypergeometric differential equation
\begin{equation}\label{3.25}
\begin{split}
2r^{2}\rho\varphi''_{l}(r^{2})=&-\left[l+\frac{n+1}{2}-\left(l+\frac{n+3}{2}-2\gamma\right)r^{2}\right]\varphi'_{l}(r^{2})+\\
&\left(l+\frac{n}{2}-\gamma\right)\left(\frac{1}{2}-\gamma\right)\varphi_{l}(r^{2}).
\end{split}
\end{equation}
Substituting (\ref{3.25}) into (\ref{3.24}), we obtain
\begin{equation*}
\begin{split}
&\left(\Delta_{g_{\mathbb{B}}}+\frac{n^{2}-1}{4}\right)
\rho^{n-s}\varphi_{l}(r^{2})r^{l}Y_{l}\\
=&\rho^{\frac{n}{2}-\gamma}\left[(1+2\gamma)(\gamma-\frac{1}{2})\rho+(\gamma^{2}-\frac{1}{4})r^{2}\right]\varphi_{l}(r^{2})r^{l}Y_{l}\\
=&\rho^{\frac{n}{2}-\gamma}(\gamma^{2}-\frac{1}{4})\varphi_{l}(r^{2})r^{l}Y_{l}.
\end{split}
\end{equation*}
This proves equality (\ref{3.18}). The roof of Theorem \ref{th3.1} is thereby completed.
\end{proof}

Before we compute the   explicit formulas of
 adapted
metric in term of $\rho$, we need the following Lemma:
\begin{lemma}\label{lm3.2}Let  $u(x)$ be the solution of  (\ref{a3.1}) and
$f=\sum\limits^{\infty}_{l=0} Y_{l}$ be the expansion in spherical harmonics.
If $\gamma\in(0,\frac{n}{2})\setminus  \frac{1}{2}\mathbb{N}$, where $\frac{1}{2}\mathbb{N}=\{0,\frac{1}{2},1,\cdots,\frac{n}{2},\cdots\}$,  then
\begin{equation}\label{3.26}
\begin{split}
u(x)
=&\rho^{n-s}\sum^{\infty}_{l=0}F(l+\frac{n}{2}-\gamma,\frac{1}{2}-\gamma,1-2\gamma;2\rho)r^{l}Y_{l}+ \\
&\rho^{s}
\frac{\Gamma(-\lambda)}{2^{2\gamma}\Gamma(\gamma)}\sum^{\infty}_{l=0}\frac{\Gamma(l+\frac{n}{2}+\gamma)}{
\Gamma(l+\frac{n}{2}-\gamma)}
F(\frac{1}{2}+\gamma,l+\frac{n}{2}+\gamma,1+2\gamma;2\rho)r^{l}Y_{l}.
\end{split}
\end{equation}
If  $\gamma=m+\frac{1}{2}$ with $ m=[\gamma]<\frac{n-1}{2}$, then
\begin{equation}\label{3.27}
\begin{split}
u(x)=\rho^{\frac{n-1}{2}-m}
\sum^{\infty}_{l=0}\left(\sum^{m}_{k=0}\frac{(l+\frac{n-1}{2}-m)_{k}(-m)_{k}}{(-2m)_{k}}\frac{(2\rho)^{k}}{k!}\right)r^{l}Y_{l},
\end{split}
\end{equation}
\end{lemma}
\begin{proof}
If $\gamma\in(0,\frac{n}{2})\setminus  \frac{1}{2}\mathbb{N}$, we have, by (\ref{2.4}) and (\ref{2.5}),
\begin{equation*}
\begin{split}
\varphi_{l}(r^{2})
=&\frac{\Gamma(\gamma+\frac{1}{2})}{\Gamma(2\gamma)}
\frac{\Gamma(l+\frac{n}{2}+\gamma)}{\Gamma(l+\frac{n+1}{2})}
F(l+\frac{n}{2}-\gamma,\frac{1}{2}-\gamma,l+\frac{n+1}{2};r^{2})\\
=&\frac{\Gamma(\gamma+\frac{1}{2})}{\Gamma(2\gamma)}
\frac{\Gamma(l+\frac{n}{2}+\gamma)}{\Gamma(l+\frac{n+1}{2})}(1-r^{2})^{2\gamma}F(\frac{1}{2}+\gamma,l+\frac{n}{2}+\gamma,l+\frac{n+1}{2};r^{2})\\
=&\frac{\Gamma(\gamma+\frac{1}{2})}{\Gamma(2\gamma)}
\frac{\Gamma(l+\frac{n}{2}+\gamma)}{\Gamma(l+\frac{n+1}{2})}(1-r^{2})^{2\gamma}\cdot\\
&\left[
\frac{\Gamma(l+\frac{n+1}{2})\Gamma(-2\gamma)}{\Gamma(l+\frac{n}{2}-\gamma)\Gamma(\frac{1}{2}-\gamma)}
F(\frac{1}{2}+\gamma,l+\frac{n}{2}+\gamma,1+2\gamma;1-r^{2})
\right.\\
&+\left.(1-r^{2})^{-2\gamma}\frac{\Gamma(l+\frac{n+1}{2})\Gamma(2\gamma)}{\Gamma(l+\frac{n}{2}+\gamma)\Gamma(\frac{1}{2}+\gamma)}
F(l+\frac{n}{2}-\gamma,\frac{1}{2}-\gamma,1-2\gamma;1-r^{2})
\right]\\
=&2^{2\gamma}\rho^{2\gamma}\frac{\Gamma(\gamma+\frac{1}{2})\Gamma(l+\frac{n}{2}+\gamma)\Gamma(-2\gamma)}{\Gamma(2\gamma)
\Gamma(l+\frac{n}{2}-\gamma)\Gamma(\frac{1}{2}-\gamma)}
F(\frac{1}{2}+\gamma,l+\frac{n}{2}+\gamma,1+2\gamma;2\rho)+\\
&F(l+\frac{n}{2}-\gamma,\frac{1}{2}-\gamma,1-2\gamma;2\rho).
\end{split}
\end{equation*}
Using the duplication formula (\ref{3.16}), we have
\begin{equation}\label{3.28}
\begin{split}
\varphi_{l}(r^{2})
=&\rho^{2\gamma}
\frac{\Gamma(-\lambda)}{2^{2\gamma}\Gamma(\gamma)}\frac{\Gamma(l+\frac{n}{2}+\gamma)}{
\Gamma(l+\frac{n}{2}-\gamma)}
F(\frac{1}{2}+\gamma,l+\frac{n}{2}+\gamma,1+2\gamma;2\rho)\\
&+F(l+\frac{n}{2}-\gamma,\frac{1}{2}-\gamma,1-2\gamma;2\rho).
\end{split}
\end{equation}
Substituting (\ref{3.28}) into (\ref{3.8}), we get (\ref{3.26}).

If  $\gamma=m+\frac{1}{2}$, then

\begin{equation*}
\begin{split}
\varphi_{l}(r)
=&\frac{\Gamma(m+1)}{\Gamma(2m+1)}
\frac{\Gamma(l+m+\frac{n+1}{2})}{\Gamma(l+\frac{n+1}{2})}
F(l+\frac{n-1}{2}-m,-m;l+\frac{n+1}{2};r)
\end{split}
\end{equation*}
is a polynomial of degree $m$ and thus
\begin{equation}\label{3.29}
\begin{split}
\varphi_{l}(r)
=&\sum^{m}_{k=0}\frac{\varphi^{(k)}_{l}(1)}{k!}(r-1)^{k}.
\end{split}
\end{equation}
Using (\ref{2.6}) and (\ref{2.4}), we have, for $0\leq k\leq m$,
\begin{equation}\label{3.30}
\begin{split}
\varphi^{(k)}_{l}(1)=&\frac{\Gamma(m+1)}{\Gamma(2m+1)}
\frac{\Gamma(l+m+\frac{n+1}{2})}{\Gamma(l+\frac{n+1}{2})}\frac{(l+\frac{n-1}{2}-m)_{k}(-m)_{k}}{(l+\frac{n+1}{2})_{k}}\cdot\\
&F(l+\frac{n-1}{2}-m+k,-m+k;l+\frac{n+1}{2}+k;1)\\
=&\frac{\Gamma(2m+1-k)}{\Gamma(2m+1)}
(l+\frac{n-1}{2}-m)_{k}(-m)_{k}\\
=&(-1)^{k}\frac{(l+\frac{n-1}{2}-m)_{k}(-m)_{k}}{(-2m)_{k}}.
\end{split}
\end{equation}
Substituting (\ref{3.30}) into (\ref{3.29}) and using (\ref{3.8}), we get (\ref{3.27}).
\end{proof}

\begin{corollary}\label{co3.3} Let $\vartheta_{s}$ be the
 solution of (\ref{a3.1}) when $f=1$.
If $\gamma\in(0,\frac{n}{2})\setminus \frac{1}{2}\mathbb{N}$, then
\begin{equation}\label{3.31}
\begin{split}
\vartheta_{s}(\rho)=&\rho^{\frac{n}{2}-\gamma}F(\frac{n}{2}-\gamma,\frac{1}{2}-\gamma;1-2\gamma;2\rho)+ \\
&\rho^{\frac{n}{2}+\gamma}\frac{\Gamma(-\lambda)}{2^{2\gamma}\Gamma(\gamma)}
\frac{\Gamma(\frac{n}{2}+\gamma)}{
\Gamma(\frac{n}{2}-\gamma)}
F(\frac{1}{2}+\gamma,\frac{n}{2}+\gamma;1+2\gamma;2\rho).
\end{split}
\end{equation}
If $\gamma=m+\frac{1}{2}$ with $ m=[\gamma]<\frac{n-1}{2}$, then
\begin{equation}\label{3.32}
\begin{split}
\vartheta_{s}(\rho)
=&\rho^{\frac{n-1}{2}-m}\sum^{m}_{k=0}\frac{(\frac{n-1}{2}-m)_{k}(-m)_{k}}{(-2m)_{k}}\frac{(2\rho)^{k}}{k!}.\\
\end{split}
\end{equation}
\end{corollary}

Now we can compute the explicit formulas of
 adapted
metric $g^{\ast}$ on the model case $(\mathbb{B}^{n+1}, \mathbb{S}^{n}, g_{\mathbb{B}})$.
\begin{proposition}
Let $\gamma\in(0,\frac{n}{2})$ and $s=\frac{n}{2}+\gamma$.  On the model case $(\mathbb{B}^{n+1}, \mathbb{S}^{n}, g_{\mathbb{B}})$ we have

(1) if  $\gamma\in(0,\frac{n}{2})\setminus\frac{1}{2}\mathbb{N}$, then $g^{\ast}=\psi_{\gamma}^{\frac{2}{n-s}}|dx|^{2}=\psi_{\gamma}^{\frac{4}{n-2\gamma}}|dx|^{2}$,
where
\begin{equation}\label{3.33}
\begin{split}
\psi_{\gamma}=&F(\frac{n}{2}-\gamma,\frac{1}{2}-\gamma,1-2\gamma;2\rho)+ \rho^{2\gamma}\frac{1}{d_{\gamma}}
\frac{\Gamma(\frac{n}{2}+\gamma)}{
\Gamma(\frac{n}{2}-\gamma)}
F(\frac{1}{2}+\gamma,\frac{n}{2}+\gamma,1+2\gamma;2\rho);
\end{split}
\end{equation}

(2)  if  $\gamma=m+\frac{1}{2}<\frac{n}{2}$ with $m=[\gamma]$ , then $g^{\ast}=\psi_{m+\frac{1}{2}}^{\frac{4}{n-2m-1}}|dx|^{2}$,
where
\begin{equation}\label{3.34}
\begin{split}
\psi_{m+\frac{1}{2}}=&\sum^{m}_{k=0}\frac{(\frac{n-1}{2}-m)_{k}(-m)_{k}}{(-2m)_{k}}\frac{(2\rho)^{k}}{k!};
\end{split}
\end{equation}

(3) if $\gamma=\frac{n}{2}$ and $n$ is an odd integer, then
\begin{equation}\label{3.35}
\begin{split}
g^{\ast}=\exp \left\{2\frac{\Gamma(\frac{n+1}{2})}{\Gamma(n)}\sum\limits^{(n-1)/2}_{k=1}\frac{\Gamma(n-k)}{\Gamma(\frac{n+1}{2}-k)k}(2\rho)^{k}
\right\}|dx|^{2}.
\end{split}
\end{equation}
\end{proposition}
\begin{proof}
By the definition of $g^{\ast}$, we have
$g^{\ast}=\vartheta^{\frac{2}{n-s}}_{s}g_{\mathbb{B}}.$ By Corollary \ref{co3.3}, we get (1) and (2).
Now we prove (\ref{3.35}). By the definition of $g^{\ast}$, we have $g^{\ast}=e^{2\tau}\rho^{-2}|dx|^{2}$,
where
\begin{equation}\label{3.36}
\begin{split}
\tau=-\frac{d}{ds}\vartheta_{s}|_{s=n}=-\lim_{s\rightarrow n}\frac{\vartheta_{s}-\vartheta_{n}}{s-n}=\lim_{\gamma\rightarrow\frac{n}{2}}
\frac{\vartheta_{s}-\vartheta_{n}}{\frac{n}{2}-\gamma}.
\end{split}
\end{equation}
By Corollary \ref{co3.3}, $\lim\limits_{s\rightarrow n}\vartheta_{s}=1$. Therefore, substituting (\ref{3.31})
into (\ref{3.36}), we have
\begin{equation}\label{3.37}
\begin{split}
\tau=&\lim_{\gamma\rightarrow\frac{n}{2}}
\frac{\rho^{\frac{n}{2}-\gamma}}{\frac{n}{2}-\gamma}\left[F(\frac{n}{2}-\gamma,\frac{1}{2}-\gamma;1-2\gamma;2\rho)-1\right]+
\lim_{\gamma\rightarrow\frac{n}{2}}
\frac{\rho^{\frac{n}{2}-\gamma}-1}{\frac{n}{2}-\gamma}\\
&\lim_{\gamma\rightarrow\frac{n}{2}}\frac{\rho^{\frac{n}{2}+\gamma}}{\frac{n}{2}-\gamma}\frac{\Gamma(-\lambda)}{2^{2\gamma}\Gamma(\gamma)}
\frac{\Gamma(\frac{n}{2}+\gamma)}{
\Gamma(\frac{n}{2}-\gamma)}
F(\frac{1}{2}+\gamma,\frac{n}{2}+\gamma;1+2\gamma;2\rho)\\
=&\lim_{\gamma\rightarrow\frac{n}{2}}
\frac{1}{\frac{n}{2}-\gamma}\left[F(\frac{n}{2}-\gamma,\frac{1}{2}-\gamma;1-2\gamma;2\rho)-1\right]+
\ln \rho\\
&\rho^{n}\frac{\Gamma(-\frac{n}{2})\Gamma(n)}{2^{n}\Gamma(\frac{n}{2})}
F(\frac{n+1}{2},n;1+n;2\rho).
\end{split}
\end{equation}
To get the last equation abouve, we use the fact $(\frac{n}{2}-\gamma)\Gamma(\frac{n}{2}-\gamma)=\Gamma(\frac{n}{2}-\gamma+1)\rightarrow1$ as $\gamma\rightarrow\frac{n}{2}$.
We compute
\begin{equation}\label{3.38}
\begin{split}
&\lim_{\gamma\rightarrow\frac{n}{2}}
\frac{1}{\frac{n}{2}-\gamma}\left[F(\frac{n}{2}-\gamma,\frac{1}{2}-\gamma;1-2\gamma;2\rho)-1\right]\\
=&\lim_{\gamma\rightarrow\frac{n}{2}}\frac{1}{\frac{n}{2}-\gamma}\sum^{\infty}_{k=1}\frac{(\frac{n}{2}-\gamma)_{k}(\frac{1}{2}-\gamma)_{k}}{(1-2\gamma)_{k}}
\frac{(2\rho)^{k}}{k!}\\
=&\sum^{(n-1)/2}_{k=1}\frac{(k-1)!(\frac{1-n}{2})_{k}}{(1-n)_{k}}\frac{(2\rho)^{k}}{k!}+\frac{1}{2}\sum^{\infty}_{k=n}\frac{(k-1)!(\frac{1-n}{2})_{\frac{n-1}{2}}
(k-\frac{n+1}{2})!}{(1-n)_{n-1}(k-n)!}\frac{(2\rho)^{k}}{k!}\\
=&\sum^{(n-1)/2}_{k=1}\frac{(k-1)!(\frac{1-n}{2})_{k}}{(1-n)_{k}}\frac{(2\rho)^{k}}{k!}+
(-1)^{\frac{n-1}{2}}\frac{2^{n-1}\Gamma(\frac{n+1}{2})}{\Gamma(n)}\rho^{n} \sum^{\infty}_{k=0}\frac{(k+\frac{n-1}{2})!}{(k+n)k!}(2\rho)^{k}.
\end{split}
\end{equation}
Using the   duplication formula (\ref{3.16}), we have
\begin{equation}\label{3.39}
\begin{split}
&(-1)^{\frac{n-1}{2}}\frac{2^{n-1}\Gamma(\frac{n+1}{2})}{\Gamma(n)}\rho^{n} \sum^{\infty}_{k=0}\frac{(k+\frac{n-1}{2})!}{(k+n)k!}(2\rho)^{k}\\
=&
(-1)^{\frac{n-1}{2}}\frac{2^{n-1}\Gamma(\frac{n+1}{2})^{2}}{\Gamma(n+1)}\rho^{n}
F(\frac{n+1}{2},n;1+n;2\rho)\\
=&(-1)^{\frac{n-1}{2}}\frac{\Gamma(\frac{1}{2})\Gamma(\frac{n+1}{2})}{\Gamma(\frac{n}{2}+1)}\rho^{n}
F(\frac{n+1}{2},n;1+n;2\rho).
\end{split}
\end{equation}
Also by the   duplication formula (\ref{3.16}), we have
\begin{equation}\label{3.40}
\begin{split}
&\rho^{n}\frac{\Gamma(-\frac{n}{2})\Gamma(n)}{2^{n}\Gamma(\frac{n}{2})}
F(\frac{n+1}{2},n;1+n;2\rho)\\
=&\rho^{n}\frac{\Gamma(-\frac{n}{2})\Gamma(\frac{n+1}{2})}{2\Gamma(\frac{1}{2})}
F(\frac{n+1}{2},n;1+n;2\rho)\\
=&(-1)^{\frac{n+1}{2}}\frac{\Gamma(\frac{1}{2})\Gamma(\frac{n+1}{2})}{\Gamma(\frac{n}{2}+1)}\rho^{n}
F(\frac{n+1}{2},n;1+n;2\rho).
\end{split}
\end{equation}
To get the last equation above, we use $\Gamma(\frac{1}{2})=\sqrt{\pi}$ and the Euler's reflection formula
\[
\Gamma\left(-\frac{n}{2}\right)\Gamma\left(1+\frac{n}{2}\right)=\frac{\pi}{-\sin\frac{n}{2}\pi}=(-1)^{\frac{n+1}{2}}\pi.
\]
Substituting (\ref{3.38}) into (\ref{3.37}) and using (\ref{3.39})-(\ref{3.40}),
we get
\begin{equation*}
\begin{split}
\tau=&\sum^{(n-1)/2}_{k=1}\frac{(k-1)!(\frac{1-n}{2})_{k}}{(1-n)_{k}}\frac{(2\rho)^{k}}{k!}+\ln\rho
=\frac{\Gamma(\frac{n+1}{2})}{\Gamma(n)}\sum\limits^{(n-1)/2}_{k=1}\frac{\Gamma(n-k)}{\Gamma(\frac{n+1}{2}-k)k}(2\rho)^{k}+\ln\rho.
\end{split}
\end{equation*}
The desired result follows.
\end{proof}

We remark
that one can also find the metric $g^{\ast}$ in dimension $2m+1$ by a  ``dimension continuity" in the spirit of the work of Branson
(see \cite{br2})),  by computing the limit
$$e^{2\tau}\rho^{-2}=\lim_{n\rightarrow2m+1}(\psi_{m+\frac{1}{2}})^{\frac{4}{n-2m-1}}.$$

\section{Proof of Theorem \ref{th1.6} and \ref{th1.7}}
In this section, we let $\gamma=m+\frac{1}{2}$ with $m=[\gamma]$. As in the proof of Theorem \ref{th3.1}, we  set
\begin{equation}\label{4.1}
\begin{split}
V_{m}(x)=&\pi^{-\frac{n}{2}}
\frac{\Gamma(\frac{n}{2}+\gamma)}{\Gamma(\gamma)}
\rho^{2m+1}\int_{\mathbb{S}^{n}}\left(\frac{1}{|x-\xi|^{2}}\right)^{\frac{n+1}{2}+m}f(\xi)d\sigma
\end{split}
\end{equation}
such that $\rho^{\frac{n-1}{2}-m}V_{m}(x)$ is the solution of (\ref{a3.1}). Moreover, if   $f=\sum\limits^{\infty}_{l=0} Y_{l}$,  then
\begin{equation}\label{4.2}
\begin{split}
V_{m}(x)=\sum^{\infty}_{l=0}\varphi_{l}(r^{2})r^{l}Y_{l},
\end{split}
\end{equation}
where
\begin{equation}\label{4.3}
\begin{split}
\varphi_{l}(r)
=&\frac{\Gamma(m+1)}{\Gamma(2m+1)}
\frac{\Gamma(l+\frac{n+1}{2}+m)}{\Gamma(l+\frac{n+1}{2})}
F(l+\frac{n-1}{2}-m,-m,l+\frac{n+1}{2};r).
\end{split}
\end{equation}
\begin{lemma}\label{lm4.1}
There holds, for $0\leq k\leq m+1$,
\begin{equation}\label{4.4}
\begin{split}
\Delta^{k} V_{m}(x)=&4^{k}\sum^{\infty}_{l=0}\frac{\Gamma(m+1)\Gamma(l+\frac{n+1}{2}+m)}{\Gamma(l+\frac{n+1}{2})\Gamma(2m+1)}
(l+\frac{n-1}{2}-m)_{k}(-m)_{k}\cdot\\
&F(l+\frac{n-1}{2}-m+k,-m+k,l+\frac{n+1}{2},r^{2})r^{l}Y_{l}.
\end{split}
\end{equation}
In particular, we have
\begin{equation}\label{4.5}
\begin{split}
\Delta^{m+1} V_{m}(x)=0.
\end{split}
\end{equation}
because of $(-m)_{m+1}=0$.
\end{lemma}
\begin{proof}
We shall prove (\ref{4.4}) by induction. It is easy to see that (\ref{4.4}) is valid for $k=0$. Now suppose that equation (\ref{4.4}) is valid for $k\geq0$.
Then we have
\begin{equation}\label{4.6}
\begin{split}
\Delta^{k+1} V_{m}(x)=&
4^{k}\sum^{\infty}_{l=0}\frac{\Gamma(m+1)\Gamma(l+\frac{n+1}{2}+m)}{\Gamma(l+\frac{n+1}{2})\Gamma(2m+1)}
(l+\frac{n-1}{2}-m)_{k}(-m)_{k}\cdot\\
&\Delta\left(F(l+\frac{n-1}{2}-m+k,-m+k,l+\frac{n+1}{2},r^{2})r^{l}Y_{l}\right).
\end{split}
\end{equation}
We compute
\begin{equation}\label{4.7}
\begin{split}
&\Delta\left(F(l+\frac{n-1}{2}-m+k,-m+k,l+\frac{n+1}{2},r^{2})r^{l}Y_{l}\right)\\
=&\left(\partial_{rr}+\frac{n}{r}\partial_{r}+\frac{1}{r^{2}}\widetilde{\Delta}\right)
\sum^{m-k}_{j=0}\frac{(l+\frac{n-1}{2}-m+k)_{j}(-m+k)_{j}}{(l+\frac{n+1}{2})_{j}}\frac{1}{j!}r^{2j+l}Y_{l}\\
=&\sum^{m-k}_{j=0}\frac{(l+\frac{n-1}{2}-m+k)_{j}(-m+k)_{j}}{(l+\frac{n+1}{2})_{j}}\frac{1}{j!}4j\left(\frac{n-1}{2}+j\right)r^{2j+l-2}Y_{l}\\
=&4(l+\frac{n-1}{2}-m+k)(-m+k)\cdot\\
&F(l+\frac{n-1}{2}-m+k+1,-m+k+1,l+\frac{n+1}{2},r^{2})r^{l}Y_{l}.
\end{split}
\end{equation}
Substituting (\ref{4.7}) into (\ref{4.6}), we have
\begin{equation*}
\begin{split}
\Delta^{k+1} V_{m}(x)=&
4^{k+1}\sum^{\infty}_{l=0}\frac{\Gamma(m+1)\Gamma(l+\frac{n+1}{2}+m)}{\Gamma(l+\frac{n+1}{2})\Gamma(2m+1)}
(l+\frac{n-1}{2}-m)_{k+1}(-m)_{k+1}\cdot\\
&F(l+\frac{n-1}{2}-m+k+1,-m+k+1,l+\frac{n+1}{2},r^{2})r^{l}Y_{l}.
\end{split}
\end{equation*}
This proves the Lemma \ref{lm4.1}.
\end{proof}

\begin{lemma}\label{lm4.2}
There holds, for $0\leq k\leq m$,
\begin{equation}\label{4.8}
\begin{split}
\Delta^{k} V_{m}|_{r=1}=&(-1)^{k}\frac{\Gamma(m+1)\Gamma(m-k+\frac{1}{2})}{\Gamma(m+\frac{1}{2})\Gamma(m-k+1)}\frac{\mathcal{P}_{2m+1}}{\mathcal{P}_{2m+1-2k}}f
\end{split}
\end{equation}
and for $0\leq k\leq m-1$,
\begin{equation}\label{4.9}
\begin{split}
\partial_{r}\Delta^{k} V_{m}|_{r=1}=&(-1)^{k+1}\frac{n-1-2m+2k}{2}
\frac{\Gamma(m+1)\Gamma(m-k+\frac{1}{2})}{\Gamma(m+\frac{1}{2})\Gamma(m-k+1)}\frac{\mathcal{P}_{2m+1}}{\mathcal{P}_{2m+1-2k}}f.
\end{split}
\end{equation}
In the case $k=m$, we have
\begin{equation}\label{4.10}
\begin{split}
\partial_{r}\Delta^{m} V_{m}|_{r=1}=&(-1)^{m}\frac{\Gamma(m+1)\Gamma(\frac{1}{2})}
{\Gamma(m+\frac{1}{2})}\left(\mathcal{P}_{2m+1}-\frac{n-1}{2}\frac{\mathcal{P}_{2m+1}}{\mathcal{P}_{1}}\right)f
\end{split}
\end{equation}
\end{lemma}
\begin{proof}
By Lemma \ref{lm4.1} and (\ref{2.3}), we have, for $0\leq k\leq m$,
\begin{equation}\label{4.11}
\begin{split}
\Delta^{k} V_{m}|_{r=1}=&4^{k}\sum^{\infty}_{l=0}\frac{\Gamma(m+1)\Gamma(l+\frac{n+1}{2}+m)}{\Gamma(l+\frac{n+1}{2})\Gamma(2m+1)}
(l+\frac{n-1}{2}-m)_{k}(-m)_{k}\cdot\\
&\frac{\Gamma(l+\frac{n+1}{2})\Gamma(2m-2k+1)}{\Gamma(m+1-k)\Gamma(l+\frac{n+1}{2}+m-k)}Y_{l}\\
=&4^{k}\frac{\Gamma(m+1)\Gamma(2m-2k+1)}{\Gamma(2m+1)\Gamma(m-k+1)}(-m)_{k}\cdot\\
&\sum^{\infty}_{l=0}\frac{\Gamma(l+\frac{n+1}{2}+m)\Gamma(l+\frac{n-1}{2}-m+k)}{\Gamma(l+\frac{n-1}{2}-m)\Gamma(l+\frac{n+1}{2}+m-k)}Y_{l}\\
=&4^{k}\frac{\Gamma(m+1)\Gamma(2m-2k+1)}{\Gamma(2m+1)\Gamma(m-k+1)}(-m)_{k}\frac{\mathcal{P}_{2m+1}}{\mathcal{P}_{2m+1-2k}}f.
\end{split}
\end{equation}
To get the last equation, we use (\ref{2.15}). Moreover, by duplication formula (\ref{3.16}), we have
\begin{equation}\label{4.12}
\begin{split}
4^{k}\frac{\Gamma(m+1)\Gamma(2m-2k+1)}{\Gamma(2m+1)\Gamma(m-k+1)}(-m)_{k}=&(-1)^{k}\frac{\Gamma(m+1)\Gamma(m-k+\frac{1}{2})}{\Gamma(m+\frac{1}{2})\Gamma(m-k+1)}.
\end{split}
\end{equation}
Substituting (\ref{4.12}) into (\ref{4.11}), we get (\ref{4.8}).
Similarly, using  (\ref{2.6}) and (\ref{2.3}), we have,
for $0\leq k\leq m-1$,
\begin{equation}\label{4.11}
\begin{split}
\partial_{r}\Delta^{k} V_{m}|_{r=1}&=4^{k}\sum^{\infty}_{l=0}\frac{\Gamma(m+1)\Gamma(l+\frac{n+1}{2}+m)}{\Gamma(l+\frac{n+1}{2})\Gamma(2m+1)}
(l+\frac{n-1}{2}-m)_{k}(-m)_{k}\cdot\\
&\left(2\frac{(l+\frac{n-1}{2}-m+k)(-m+k)}{l+\frac{n+1}{2}}\frac{\Gamma(l+\frac{n+1}{2}+1)\Gamma(2m-2k)}{\Gamma(m+1-k)\Gamma(l+\frac{n+1}{2}+m-k)}
\right.\\
&+\left.l\frac{\Gamma(l+\frac{n+1}{2})\Gamma(2m-2k+1)}{\Gamma(m+1-k)\Gamma(l+\frac{n+1}{2}+m-k)}\right)Y_{l}\\
&=-\left(\frac{n-1}{2}-m+k\right)\Delta^{k} V_{m}|_{r=1}.
\end{split}
\end{equation}
These prove (\ref{4.9}).
Now we  prove (\ref{4.10}). By Lemma \ref{lm4.1},
\begin{equation*}
\begin{split}
\Delta^{m} V_{m}(x)=&4^{k}\frac{\Gamma(m+1)}{\Gamma(2m+1)}(-m)_{m}\sum^{\infty}_{l=0}
\frac{\Gamma(l+\frac{n+1}{2}+m)}{(l+\frac{n-1}{2})\Gamma(l+\frac{n-1}{2}-m)}r^{l}Y_{l}.
\end{split}
\end{equation*}
Therefore, by (\ref{2.15}), we have
\begin{equation}\label{4.13}
\begin{split}
\partial_{r}\Delta^{m} V_{m}|_{r=1}=&4^{m}\frac{\Gamma(m+1)}{\Gamma(2m+1)}(-m)_{m}\sum^{\infty}_{l=0}\frac{l}{(l+\frac{n-1}{2})}
\frac{\Gamma(l+\frac{n+1}{2}+m)}{\Gamma(l+\frac{n-1}{2}-m)}Y_{l}\\
=&4^{m}\frac{\Gamma(m+1)}{\Gamma(2m+1)}(-m)_{m}\sum^{\infty}_{l=0}\left(1-\frac{\frac{n-1}{2}}{l+\frac{n-1}{2}}\right)\cdot\\
&\frac{\Gamma(l+\frac{n+1}{2}+m)}{\Gamma(l+\frac{n-1}{2}-m)}Y_{l}\\
=&(-1)^{m}4^{m}\frac{\Gamma(m+1)^{2}}{\Gamma(2m+1)}\left(\mathcal{P}_{2m+1}-\frac{n-1}{2}\frac{\mathcal{P}_{2m+1}}{\mathcal{P}_{1}}\right)f\\
=&(-1)^{m}\frac{\Gamma(m+1)\Gamma(\frac{1}{2})}{\Gamma(m+\frac{1}{2})}\left(\mathcal{P}_{2m+1}-\frac{n-1}{2}\frac{\mathcal{P}_{2m+1}}{\mathcal{P}_{1}}\right)f.
\end{split}
\end{equation}
To get the last equation, we use duplication formula (\ref{3.16}). The desired result follows.
\end{proof}
\

\textbf{Proof of Theorem \ref{th1.6}}
We firstly prove (\ref{1.18}) when $v=V_{m}$. By Lemma \ref{lm4.2}
and  Green's formula (see e.g. \cite{evans}, Appendix C), we have
\begin{equation}\label{4.15}
\begin{split}
0=&\int_{\mathbb{B}^{n+1}}V_{m}\Delta^{m+1}V_{m}dx\\
=&\int_{\mathbb{B}^{n+1}}\Delta V_{m}\Delta^{m}V_{m}dx+\int_{\mathbb{S}^{n}}V_{m}\partial_{r}\Delta^{m}V_{m}d\sigma-
\int_{\mathbb{S}^{n}}\partial_{r}V_{m}\Delta^{m}V_{m}d\sigma\\
=&\int_{\mathbb{B}^{n+1}}\Delta V_{m}\Delta^{m}V_{m}dx+(-1)^{m}\frac{\Gamma(m+1)\Gamma(\frac{1}{2})}{\Gamma(m+\frac{1}{2})}
\int_{\mathbb{S}^{n}}
f\mathcal{P}_{2m+1}fd\sigma-\\
&(-1)^{m}\frac{n-1}{2}\frac{\Gamma(m+1)\Gamma(\frac{1}{2})}{\Gamma(m+\frac{1}{2})}\int_{\mathbb{S}^{n}}
f\frac{\mathcal{P}_{2m+1}}{\mathcal{P}_{1}}fd\sigma.
\end{split}
\end{equation}
If $m$ is an odd integer, then by Lemma \ref{lm4.2} and
and  Green's formula, we have
\begin{equation}\label{4.16}
\begin{split}
\int_{\mathbb{B}^{n+1}}\Delta V_{m}\Delta^{m}V_{m}dx=&\int_{\mathbb{B}^{n+1}}|\Delta^{\frac{m+1}{2}}V_{m}|^{2}dx
+\sum^{\frac{m-1}{2}}_{k=1}(m-2k)
\frac{\Gamma(m+1)^{2}}{\Gamma(m+\frac{1}{2})^{2}}\cdot\\
&\frac{\Gamma(k+\frac{1}{2})\Gamma(m-k+\frac{1}{2})}{\Gamma(k+1)\Gamma(m-k+1)}
\int_{\mathbb{S}^{n}}
f\frac{\mathcal{P}_{2m+1}^{2}}{\mathcal{P}_{2m+1-2k}\mathcal{P}_{2k+1}}fd\sigma.
\end{split}
\end{equation}
If $m$ is a even integer, then also by Lemma \ref{lm4.2} and
and  Green's formula, we have
\begin{equation}\label{4.17}
\begin{split}
\int_{\mathbb{B}^{n+1}}\Delta V_{m}\Delta^{m}V_{m}dx=&-\int_{\mathbb{B}^{n+1}}|\nabla\Delta^{\frac{m}{2}}V_{m}|^{2}dx
-\sum^{\frac{m-2}{2}}_{k=1}(m-2k)
\frac{\Gamma(m+1)^{2}}{\Gamma(m+\frac{1}{2})^{2}}\cdot\\
&\frac{\Gamma(k+\frac{1}{2})\Gamma(m-k+\frac{1}{2})}{\Gamma(k+1)\Gamma(m-k+1)}
\int_{\mathbb{S}^{n}}
f\frac{\mathcal{P}_{2m+1}^{2}}{\mathcal{P}_{2m+1-2k}\mathcal{P}_{2k+1}}fd\sigma-\\
&\frac{n-1-m}{2}\left(\frac{\Gamma(m+1)\Gamma(\frac{m+1}{2})}{\Gamma(m+\frac{1}{2})\Gamma(\frac{m}{2}+1)}\right)^{2}\int_{\mathbb{S}^{n}}
f\frac{\mathcal{P}_{2m+1}^{2}}{\mathcal{P}_{m+1}^{2}}fd\sigma.
\end{split}
\end{equation}
Substituting (\ref{4.16}) and (\ref{4.17}) into (\ref{4.15}), we have
\begin{equation}\label{4.18}
\begin{split}
\frac{\Gamma(m+1)\Gamma(\frac{1}{2})}{\Gamma(m+\frac{1}{2})}
\int_{\mathbb{S}^{n}}
f\mathcal{P}_{2m+1}fd\sigma=\int_{\mathbb{B}^{n+1}}|\nabla^{m+1}V_{m}|^{2}dx+
\int_{\mathbb{S}^{n}}
f \mathcal{T}_{m}fd\sigma,
\end{split}
\end{equation}
where $\mathcal{T}_{m}$ is defined in Theorem (\ref{th1.6}). Therefore, by Theorem \ref{be1}, we prove the Theorem \ref{1.6}
when $v=V_{m}$ and in this case the only extremal functions is given by (\ref{b1.19}). For general $v$ with  the Neumann boundary condition (\ref{1.17}), we claim
\begin{equation*}
\begin{split}
\int_{\mathbb{B}^{n+1}}|\nabla^{m+1}V_{m}|^{2}dx\leq \int_{\mathbb{B}^{n+1}}|\nabla^{m+1}v|^{2}dx.
\end{split}
\end{equation*}
In fact, we have
\begin{equation*}
\begin{split}
0\leq&\int_{\mathbb{B}^{n+1}}|\nabla^{m+1}(v-U_{m})|^{2}dx\\
=&\int_{\mathbb{B}^{n+1}}|\nabla^{m+1}v|^{2}dx-\int_{\mathbb{B}^{n+1}}|\nabla^{m+1}U_{m}|^{2}dx\\
&-2\int_{\mathbb{B}^{n+1}}\nabla^{m+1}(v-U_{m})\cdot\nabla^{m+1}U_{m}dx.
\end{split}
\end{equation*}
Since $v$ and $U_{m}$ have the same Neumann boundary condition (\ref{1.17}), we have
$$\int_{\mathbb{B}^{n+1}}\nabla^{m+1}(v-U_{m})\cdot\nabla^{m+1}U_{m}dx
=(-1)^{m+1}\int_{\mathbb{B}^{n+1}}(v-U_{m})\Delta^{m+1}U_{m}dx=0.$$
Therefore,
\begin{equation}\label{b4.19}
\begin{split}
\int_{\mathbb{B}^{n+1}}|\nabla^{m+1}v|^{2}dx=& \int_{\mathbb{B}^{n+1}}|\nabla^{m+1}(v-U_{m})|^{2}dx+\int_{\mathbb{B}^{n+1}}|\nabla^{m+1}U_{m}|^{2}dx\\
\geq&\int_{\mathbb{B}^{n+1}}|\nabla^{m+1}U_{m}|^{2}dx.
\end{split}
\end{equation}
These prove the claim.

Finally, we prove the uniqueness of the extremal functions. If $v$ is any extremal function with the  Neumann boundary condition (\ref{1.17}), then by (\ref{b4.19})
it must satisfies $\nabla^{m+1}(v-U_{m})=0$ and thus 
$\Delta^{m+1}v=\Delta^{m+1}U_{m}=0.$  By the uniqueness of the solution, we get $v=V_{m}(x)$.
Thus, by Theorem \ref{be1},  the only extremal function is that given by (\ref{b1.19}).
The proof of Theorem \ref{th1.6}
is thereby completed.\
\\
\
\

\textbf{Proof of Theorem \ref{th1.7}} With the same argument in the proof of Theorem \ref{th1.6}, we need only consider the case $v=V_{m}(x)$
with $m=\frac{n-1}{2}$. Using (\ref{4.15})-(\ref{4.18}), we get
\begin{equation}\label{4.20}
\begin{split}
\frac{\Gamma(\frac{n+1}{2})\Gamma(\frac{1}{2})}{\Gamma(\frac{n}{2})}
\int_{\mathbb{S}^{n}}
f\mathcal{P}_{n}fd\sigma=\int_{\mathbb{B}^{n+1}}|\nabla^{\frac{n+1}{2}}V_{m}|^{2}dx+
\int_{\mathbb{S}^{n}}
f \mathcal{T}_{\frac{n-1}{2}}fd\sigma.
\end{split}
\end{equation}
Therefore, by Theorem \ref{be1}, we have
\begin{equation}\label{4.21}
\begin{split}
&\ln\left(
\frac{1}{\omega_{n}}\int_{\mathbb{S}^{n}}e^{n(f-\overline{f})}d\sigma\right)\\
\leq&\frac{n}{2(n-1)!\omega_{n}}
\int_{\mathbb{S}^{n}}f \mathcal{P}_{n}fd\sigma\\
\leq&\frac{n}{2(n-1)!\omega_{n}}\frac{\Gamma(\frac{n}{2})}
{\Gamma(\frac{n+1}{2})\Gamma(\frac{1}{2})}\left(\int_{\mathbb{B}^{n+1}}|\nabla^{\frac{n+1}{2}}V_{m}|^{2}dx+
\int_{\mathbb{S}^{n}}
f \mathcal{T}_{\frac{n-1}{2}}fd\sigma\right)\\
=&\frac{n}{2^{n+1}\pi^{\frac{n+1}{2}}\Gamma(\frac{n+1}{2})}
\left(\int_{\mathbb{B}^{n+1}}|\nabla^{\frac{n+1}{2}} V_{m}|^{2}dx+\int_{\mathbb{S}^{n}}f\mathcal{T}_{\frac{n-1}{2}}fd\sigma\right).
\end{split}
\end{equation}
To get the last equation, we use the fact $\omega_{n}=\frac{2\pi^{\frac{n+1}{2}}}{\Gamma(\frac{n+1}{2})}=\frac{2^{n}\Gamma(\frac{n}{2})\pi^{\frac{n}{2}}}{\Gamma(n)}$.

With the same argument in Theorem \ref{th1.6} and using Theorem \ref{be1}, we have get the only 
only extremal function is that given by (\ref{b1.20}). The proof of Theorem \ref{th1.7} is thereby completed.

\section{Proof of Theorem \ref{th1.8}}
Since the M\"obius transform
$\mathcal{M}: (\mathbb{R}^{n+1}_{+},g_{\mathbb{H}})\rightarrow(\mathbb{B}^{n+1}, g_{\mathbb{B}})$, where $g_{\mathbb{H}}=\frac{|dx|^{2}+|dy|^{2}}{y^{2}}$, defined by
$$\mathcal{M}(x,y)=\left(\frac{2x}{(1+y)^{2}+|x|^{2}},
\frac{1-|x|^{2}-y^{2}}{(1+y)^{2}+|x|^{2}}\right),$$
is an isometry between the two models of hyperbolic space, by Theorem \ref{th3.1}, we can solve the Poisson equation (\ref{3.1}) on $(\mathbb{R}^{n+1}_{+},g_{\mathbb{H}})$:
\begin{theorem}\label{th5.1}
Let $\gamma\in (0,\frac{n}{2})$ and $s=\frac{n}{2}+\gamma$. The solution of the Poisson equation on the hyperbolic space $(\mathbb{R}^{n+1}_{+},g_{\mathbb{H}})$
\begin{equation}\label{5.1}
\begin{split}
\left\{
  \begin{array}{ll}
    -\Delta_{g_{\mathbb{H}}}u-s(n-s)u=0 & \hbox{in $\mathbb{R}^{n+1}$,} \\
    u=Fy^{n-s}+Hy^{s}, & \hbox{} \\
 F|_{\partial \mathbb{R}_{+}^{n+1}}=f(x), & \hbox{.}
  \end{array}
\right.
\end{split}
\end{equation}
is
\begin{equation}\label{5.2}
\begin{split}
u(x,y)=\pi^{-\frac{n}{2}}
\frac{\Gamma(\frac{n}{2}+\gamma)}{\Gamma(\gamma)}
\int_{\mathbb{R}^{n}}\left(\frac{y}{|x-\xi|^{2}+y^{2}}\right)^{s}f(\xi)d\xi,
\end{split}
\end{equation}
where $f$ and its derivatives have fast decay at infinity (for example, in certain fractional Sobolev spaces).
\end{theorem}
Now we let $\gamma=m+\frac{1}{2}$ with $m$ an integer. In this case, we can find out the relationship between  the kernel in (\ref{5.2})
and the Poisson kernel. In fact, we have the following:
\begin{lemma}\label{lm5.2} There holds, for $(x,y)\in\mathbb{R}^{n+1}_{+}$ and $m\in\mathbb{N}$,
\begin{equation}\label{5.3}
\begin{split}
&\left(\sum^{m}_{k=0}\frac{2^{k}}{k!}\frac{\Gamma(2m-k+1)}{\Gamma(m-k+1)}(-y)^{k}\frac{d^{k}}{dy^{k}}\right)
\frac{y}{(|x|^{2}+y^{2})^{\frac{n+1}{2}}}\\
=&2^{2m}\frac{\Gamma(m+\frac{n+1}{2})}{\Gamma(\frac{n+1}{2})}
\frac{y^{1+2m}}{(|x|^{2}+y^{2})^{\frac{n+1}{2}+m}}.
\end{split}
\end{equation}
\end{lemma}
\begin{proof}
We shall prove it by induction. It is easy to see (\ref{5.3}) is valid for $m=0$.
Suppose (\ref{5.3}) is valid for $m\geq0$. We compute
\begin{equation}\label{5.4}
\begin{split}
y\frac{d}{dy}
\left(\frac{y^{1+2m}}{(|x|^{2}+y^{2})^{\frac{n+1}{2}+m}}\right)=&(1+2m)\frac{y^{1+2m}}{(|x|^{2}+y^{2})^{\frac{n+1}{2}+m}}-\\
&(n+1+2m)
\frac{y^{3+2m}}{(|x|^{2}+y^{2})^{\frac{n+1}{2}+m+1}}.
\end{split}
\end{equation}
Therefore,
\begin{equation}\label{5.5}
\begin{split}
&(n+1+2m)\frac{y^{3+2m}}{(|x|^{2}+y^{2})^{\frac{n+1}{2}+m+1}}
=\left(1+2m-y\frac{d}{dy}\right)
\frac{y^{1+2m}}{(|x|^{2}+y^{2})^{\frac{n+1}{2}+m}}\\
=&2^{-2m}\frac{\Gamma(\frac{n+1}{2})}
{\Gamma(m+\frac{n+1}{2})}\cdot\\&
\left(1+2m-y\frac{d}{dy}\right)
\left(\sum^{m}_{k=0}\frac{2^{k}}{k!}\frac{\Gamma(2m-k+1)}{\Gamma(m-k+1)}(-y)^{k}\frac{d^{k}}{dy^{k}}\right)
\frac{y}{(|x|^{2}+y^{2})^{\frac{n+1}{2}}}.
\end{split}
\end{equation}
Since
\begin{equation}\label{5.6}
\begin{split}
&\left(1+2m-y\frac{d}{dy}\right)
\left(\sum^{m}_{k=0}\frac{2^{k}}{k!}\frac{\Gamma(2m-k+1)}{\Gamma(m-k+1)}(-y)^{k}\frac{d^{k}}{dy^{k}}\right)\\
=&\sum^{m}_{k=0}\frac{2^{k}}{k!}\frac{\Gamma(2m-k+1)}{\Gamma(m-k+1)}(1+2m-k)(-y)^{k}\frac{d^{k}}{dy^{k}}+
\\
&\sum^{m}_{k=0}\frac{2^{k}}{k!}\frac{\Gamma(2m-k+1)}{\Gamma(m-k+1)}(-y)^{k+1}\frac{d^{k+1}}{dy^{k+1}}\\
=&\sum^{m}_{k=0}\frac{2^{k}}{k!}\frac{\Gamma(2m-k+1)}{\Gamma(m-k+1)}(1+2m-k)(-y)^{k}\frac{d^{k}}{dy^{k}}+
\\
&\sum^{m+1}_{k=1}\frac{2^{k-1}}{(k-1)!}\frac{\Gamma(2m-k+2)}{\Gamma(m-k+2)}(-y)^{k}\frac{d^{k}}{dy^{k}}\\
=&\frac{1}{2}\left(\sum^{m+1}_{k=0}\frac{2^{k}}{k!}\frac{\Gamma(2m-k+1)}{\Gamma(m-k+1)}(-y)^{k}\frac{d^{k}}{dy^{k}}\right)
\frac{y}{(|x|^{2}+y^{2})^{\frac{n+1}{2}}},
\end{split}
\end{equation}
We have, by (\ref{5.5}) and (\ref{5.6}),
\begin{equation*}
\begin{split}
&\left(\sum^{m+1}_{k=0}\frac{2^{k}}{k!}\frac{\Gamma(2m-k+3)}{\Gamma(m-k+2)}(-y)^{k}\frac{d^{k}}{dy^{k}}\right)
\frac{y}{(|x|^{2}+y^{2})^{\frac{n+1}{2}}}\\
=&2^{2m+2}\frac{\Gamma(m+1+\frac{n+1}{2})}{\Gamma(\frac{n+1}{2})}
\frac{y^{3+2m}}{(|x|^{2}+y^{2})^{\frac{n+1}{2}+m+1}}.
\end{split}
\end{equation*}
These completes the proof of Lemma \ref{lm5.2}.
\end{proof}

In the rest paper, we let $\Delta_{x}=\sum\limits^{n}_{i=1}\partial_{x_{i}x_{i}}$ and $\Delta=\Delta_{x}+\partial_{yy}$.
Since the Poisson kernel $e^{-y\sqrt{-\Delta_{x}}}$ on $\mathbb{R}_{+}^{n+1}$ is given by (see e.g. \cite{sw})
\begin{equation}\label{5.7}
\begin{split}
e^{-y\sqrt{-\Delta_{x}}}=\frac{\pi^{-\frac{n}{2}}\Gamma(\frac{n+1}{2})}{\Gamma(\frac{1}{2})}\frac{y}{(|x|^{2}+y^{2})^{\frac{n+1}{2}}},
\end{split}
\end{equation}
we have, by Theorem \ref{th5.1} and Lemma \ref{lm5.2}, that the solution of (\ref{5.1}) is
\begin{equation}\label{5.8}
\begin{split}
u(x,y)=&y^{\frac{n-1}{2}-m}\frac{2^{-2m}\Gamma(\frac{1}{2})}{\Gamma(m+\frac{1}{2})}\left(\sum^{m}_{k=0}\frac{2^{k}}{k!}
\frac{\Gamma(2m-k+1)}{\Gamma(m-k+1)}(-y)^{k}\frac{d^{k}}{dy^{k}}\right)e^{-y\sqrt{-\Delta_{x}}}f\\
=&y^{\frac{n-1}{2}-m}\frac{2^{-2m}\Gamma(\frac{1}{2})}{\Gamma(m+\frac{1}{2})}\sum^{m}_{k=0}\frac{2^{k}}{k!}
\frac{\Gamma(2m-k+1)}{\Gamma(m-k+1)}y^{k}(-\Delta_{x})^{\frac{k}{2}}e^{-y\sqrt{-\Delta_{x}}}f.
\end{split}
\end{equation}

\begin{lemma}\label{lm5.3}
Let $U_{m}(x,y)=u(x,y)y^{-\frac{n-1}{2}+m}$, where $u(x,y)$ is  given by (\ref{5.8}). Then for $0\leq k\leq m+1$, we have
\begin{equation}\label{5.9}
\begin{split}
\Delta^{k}U_{m}(x,y)=&(-1)^{k}\frac{2^{2k-2m}\Gamma(m+1)\Gamma(\frac{1}{2})}{\Gamma(m-k+1)\Gamma(m+\frac{1}{2})}\cdot\\
&\sum^{m-k}_{j=0}\frac{2^{j}}{j!}\frac{\Gamma(2m-2k-j+1)}{\Gamma(m-k-j+1)}
 y^{j}(-\Delta_{x})^{k+\frac{j}{2}}e^{-y\sqrt{-\Delta_{x}}}f.
\end{split}
\end{equation}
In particular,  $\Delta^{m+1}u(x,y)=0$ because of $(-m)_{m+1}=0$. Moreover,
\begin{equation}\label{5.10}
\begin{split}
\Delta^{k}U_{m}(x,y)|_{y=0}=&\frac{\Gamma(m+1)\Gamma(m+\frac{1}{2}-k)}{\Gamma(m-k-1)\Gamma(m+\frac{1}{2})}\Delta^{k}_{x}f,
\;\;\;0\leq k\leq m;\\
\partial_{y}\Delta^{k}U_{m}(x,y)|_{y=0}=&0,\;\;\;\;\;\;\;\;\;\;\;\;\;\;\;\;\;\;\;\;\;\;\;\;\;\;\;\;\;\;\;\;\;\;\;\;\;\;
\;\;\;\;\;\;\;\;0\leq k\leq m-1;\\
\partial_{y}\Delta^{m}U_{m}(x,y)|_{y=0}=&(-1)^{m}\frac{\Gamma(m+1)\Gamma(\frac{1}{2})}{\Gamma(m+\frac{1}{2})}(-\Delta_{x})^{m+\frac{1}{2}}f.
\end{split}
\end{equation}
\end{lemma}
\begin{proof}
We prove (\ref{5.9}) by induction. Obviously,  (\ref{5.9}) is valid for $k=0$. Suppose (\ref{5.9}) is valid for $k$. Then
\begin{equation*}
\begin{split}
&\Delta^{k+1}U_{m}(x,y)=(-1)^{k}\frac{2^{2k-2m}\Gamma(m+1)\Gamma(\frac{1}{2})}{\Gamma(m-k+1)\Gamma(m+\frac{1}{2})}\cdot\\
&\sum^{m-k}_{j=0}\frac{2^{j}}{j!}\frac{\Gamma(2m-2k-j+1)}{\Gamma(m-k-j+1)}
 \Delta\left(y^{j}(-\Delta_{x})^{k+\frac{j}{2}}e^{-y\sqrt{-\Delta_{x}}}f\right)\\
=&(-1)^{k}\frac{2^{2k-2m}\Gamma(m+1)\Gamma(\frac{1}{2})}{\Gamma(m-k+1)\Gamma(m+\frac{1}{2})}
\sum^{m-k}_{j=0}\frac{2^{j}}{j!}\frac{\Gamma(2m-2k-j+1)}{\Gamma(m-k-j+1)}\cdot\\
&
\left(j(j-1)y^{j-2}(-\Delta_{x})^{\frac{j}{2}}-2jy^{j}(-\Delta_{x})^{\frac{1+j}{2}}\right)(-\Delta_{x})^{k}e^{-y\sqrt{-\Delta_{x}}}f\\
=&(-1)^{k}\frac{2^{2k-2m}\Gamma(m+1)\Gamma(\frac{1}{2})}{\Gamma(m-k+1)\Gamma(m+\frac{1}{2})}
\left(\sum^{m-k}_{j=2}\frac{2^{j}}{(j-2)!}\frac{\Gamma(2m-2k-j+1)}{\Gamma(m-k-j+1)}(-\Delta_{x})^{\frac{j}{2}}\right.\\
&-2\left.\sum^{m-k}_{j=1}\frac{2^{j}}{(j-1)!}\frac{\Gamma(2m-2k-j+1)}{\Gamma(m-k-j+1)}(-\Delta_{x})^{\frac{1+j}{2}}  \right)
(-\Delta_{x})^{k}e^{-y\sqrt{-\Delta_{x}}}f.
\end{split}
\end{equation*}
A simple calculation shows 
\begin{equation*}
\begin{split}
&\sum^{m-k}_{j=2}\frac{2^{j}}{(j-2)!}\frac{\Gamma(2m-2k-j+1)}{\Gamma(m-k-j+1)}(-\Delta_{x})^{\frac{j}{2}}
-2\sum^{m-k}_{j=1}\frac{2^{j}}{(j-1)!}\frac{\Gamma(2m-2k-j+1)}{\Gamma(m-k-j+1)}(-\Delta_{x})^{\frac{1+j}{2}}\\
=&\sum^{m-k-2}_{j=0}\frac{2^{j+2}}{j!}\frac{\Gamma(2m-2k-j-1)}{\Gamma(m-k-j-1)}(-\Delta_{x})^{1+\frac{j}{2}}
-2\sum^{m-k-1}_{j=0}\frac{2^{j+4}}{j!}\frac{\Gamma(2m-2k-j)}{\Gamma(m-k-j)}(-\Delta_{x})^{1+\frac{j}{2}}\\
=&4(k-m)\sum^{m-k-1}_{j=0}\frac{2^{j}}{j!}\frac{\Gamma(2m-2k-j-1)}{\Gamma(m-k-j)}
 y^{j}(-\Delta_{x})^{k+1+\frac{j}{2}}e^{-y\sqrt{-\Delta_{x}}}f.
\end{split}
\end{equation*}
and thus we have
\begin{equation*}
\begin{split}
\Delta^{k+1}U_{m}(x,y)=&
(-1)^{k+1}\frac{2^{2k+2-2m}\Gamma(m+1)\Gamma(\frac{1}{2})}{\Gamma(m-k)\Gamma(m+\frac{1}{2})}\cdot\\
&\sum^{m-k-1}_{j=0}\frac{2^{j}}{j!}\frac{\Gamma(2m-2k-j-1)}{\Gamma(m-k-j)}
 y^{j}(-\Delta_{x})^{k+1+\frac{j}{2}}e^{-y\sqrt{-\Delta_{x}}}f
\end{split}
\end{equation*}
These prove (\ref{5.9}). 

By (\ref{5.9}) and (\ref{3.16}), we have
\begin{equation*}
\begin{split}
\Delta^{k}U_{m}(x,y)|_{y=0}=&(-1)^{k}\frac{2^{2k-2m}\Gamma(m+1)\Gamma(\frac{1}{2})}{\Gamma(m-k+1)\Gamma(m+\frac{1}{2})}
\frac{\Gamma(2m-2k+1)}{\Gamma(m-k+1)}(-\Delta)^{k}_{x}f\\
=&\frac{\Gamma(m+1)\Gamma(m+\frac{1}{2}-k)}{\Gamma(m-k-1)\Gamma(m+\frac{1}{2})}\Delta^{k}_{x}f.
\end{split}
\end{equation*}
Similarly, 
$\partial_{y}\Delta^{k}U_{m}(x,y)|_{y=0}=0$ when  $0\leq k\leq m-1$ and
$$\partial_{y}\Delta^{m}U_{m}(x,y)|_{y=0}=(-1)^{m}\frac{\Gamma(m+1)\Gamma(\frac{1}{2})}{\Gamma(m+\frac{1}{2})}(-\Delta_{x})^{m+\frac{1}{2}}f.$$
These complete the  proof of Lemma \ref{lm5.3}
\end{proof}

We can also compute the Neumann boundary condition $\partial^{k}_{y}U_{m}|_{y=0}$ for $0\leq k\leq 2m$.
In fact, we have the following:
\begin{lemma}\label{lm5.4}
There holds, for $0\leq k\leq m$,
\begin{equation}\label{5.11}
\begin{split}
\partial^{2k}_{y}U_{m}(x,y)|_{y=0}=&\frac{\Gamma(k+\frac{1}{2})\Gamma(m-k+\frac{1}{2})}{\Gamma(\frac{1}{2})\Gamma(m+\frac{1}{2})}\Delta_{x}^{k}f
\end{split}
\end{equation}
and for $0\leq k\leq m-1$,
\begin{equation}\label{5.12}
\begin{split}
\partial^{2k+1}_{y}U_{m}(x,y)|_{y=0}=&0.
\end{split}
\end{equation}

\end{lemma}
\begin{proof}
Firstly, we prove (\ref{5.11}) by induction. It is easy to see (\ref{5.11}) is valid for $k=0$.
Suppose that (\ref{5.11}) is valid for $k$, i.e.,
\begin{equation}\label{5.13}
\begin{split}
&\lim_{y\rightarrow0+}\partial^{2k}_{y}\left(\frac{\Gamma(m+1)}{\Gamma(2m+1)}\sum^{m}_{j=0}\frac{2^{j}}{j!}
\frac{\Gamma(2m-j+1)}{\Gamma(m-j+1)}y^{j}(-\Delta_{x})^{\frac{j}{2}}e^{-y\sqrt{-\Delta_{x}}}f\right)
\\=&\frac{\Gamma(k+\frac{1}{2})\Gamma(m-k+\frac{1}{2})}{\Gamma(\frac{1}{2})\Gamma(m+\frac{1}{2})}\Delta_{x}^{k}f.
\end{split}
\end{equation}
Replacing $m$ by $m-1$ in (\ref{5.13}), we have
\begin{equation}\label{5.14}
\begin{split}
&\lim_{y\rightarrow0+}\partial^{2k}_{y}\left(\frac{\Gamma(m)}{\Gamma(2m-1)}\sum^{m-1}_{j=0}\frac{2^{j}}{j!}
\frac{\Gamma(2m-j-1)}{\Gamma(m-j)}y^{j}(-\Delta_{x})^{\frac{j}{2}}e^{-y\sqrt{-\Delta_{x}}}f\right)
\\=&\frac{\Gamma(k+\frac{1}{2})\Gamma(m-k-\frac{1}{2})}{\Gamma(\frac{1}{2})\Gamma(m-\frac{1}{2})}\Delta_{x}^{k}f.
\end{split}
\end{equation}
By (\ref{5.9}), we have
\begin{equation}\label{5.15}
\begin{split}
&\partial^{2k+2}_{y}U_{m}(x,y)=\partial^{2k}_{y}\left[\Delta U_{m}(x,y)-\Delta_{x}U_{m}(x,y)\right]\\
=&\partial^{2k}_{y}\left[-\frac{m}{m-\frac{1}{2}}\frac{\Gamma(m)}{\Gamma(2m-1)}
\sum^{m-1}_{j=0}\frac{2^{j}}{j!}\frac{\Gamma(2m-2k-j+1)}{\Gamma(m-k-j+1)}
 y^{j}(-\Delta_{x})^{1+\frac{j}{2}}e^{-y\sqrt{-\Delta_{x}}}f\right.\\
&\left.-\frac{\Gamma(m+1)}{\Gamma(2m+1)}\sum^{m}_{j=0}\frac{2^{j}}{j!}
\frac{\Gamma(2m-j+1)}{\Gamma(m-j+1)}y^{j}(-\Delta_{x})^{\frac{j}{2}}e^{-y\sqrt{-\Delta_{x}}}\Delta_{x}f\right]\\
=&\partial^{2k}_{y}\left[\frac{m}{m-\frac{1}{2}}\frac{\Gamma(m)}{\Gamma(2m-1)}
\sum^{m-1}_{j=0}\frac{2^{j}}{j!}\frac{\Gamma(2m-2k-j+1)}{\Gamma(m-k-j+1)}
 y^{j}(-\Delta_{x})^{\frac{j}{2}}e^{-y\sqrt{-\Delta_{x}}}\Delta_{x}f\right.\\
&\left.-\frac{\Gamma(m+1)}{\Gamma(2m+1)}\sum^{m}_{j=0}\frac{2^{j}}{j!}
\frac{\Gamma(2m-j+1)}{\Gamma(m-j+1)}y^{j}(-\Delta_{x})^{\frac{j}{2}}e^{-y\sqrt{-\Delta_{x}}}\Delta_{x}f\right].
\end{split}
\end{equation}
Combing (\ref{5.13})-(\ref{5.15}) yields
\begin{equation*}
\begin{split}
\lim_{y\rightarrow0+}\partial^{2k+2}_{y}U_{m}(x,y)=&\frac{m}{m-\frac{1}{2}}
\frac{\Gamma(k+\frac{1}{2})\Gamma(m-k-\frac{1}{2})}{\Gamma(\frac{1}{2})\Gamma(m-\frac{1}{2})}\Delta_{x}^{k+1}f-
\frac{\Gamma(k+\frac{1}{2})\Gamma(m-k+\frac{1}{2})}{\Gamma(\frac{1}{2})\Gamma(m+\frac{1}{2})}\Delta_{x}^{k+1}f\\
=&\frac{\Gamma(k+1+\frac{1}{2})\Gamma(m-k-\frac{1}{2})}{\Gamma(\frac{1}{2})\Gamma(m+\frac{1}{2})}\Delta_{x}^{k+1}f.
\end{split}
\end{equation*}
These prove (\ref{5.11}).
The proof of (\ref{5.12}) is completely analogous to that of (\ref{5.11}) and we omit.
\end{proof}
\

\textbf{Proof of  Theorem \ref{th1.8}}
With the same argument in the proof of Theorem \ref{th1.6}, we need only consider the case $u=U_{m}(x,y)$.
By Lemma \ref{lm5.3}
and  Green's formula, we have
\begin{equation*}
\begin{split}
0=&\int_{\mathbb{R}_{+}^{n+1}}U_{m}\Delta^{m+1}U_{m}dxdy\\
=&\int_{\mathbb{R}_{+}^{n+1}}\Delta U_{m}\Delta^{m}U_{m}dxdy+\int_{\mathbb{R}^{n}}U_{m}\partial_{y}\Delta^{m}U_{m}dx-
\int_{\mathbb{R}^{n}}\partial_{y}U_{m}\Delta^{m}U_{m}dx\\
=&\int_{\mathbb{R}_{+}^{n+1}}\Delta U_{m}\Delta^{m}U_{m}dxdy+(-1)^{m}\frac{\Gamma(m+1)\Gamma(\frac{1}{2})}{\Gamma(m+\frac{1}{2})}
\int_{\mathbb{R}^{n}}
U_{m}(x,0)(-\Delta_{x})^{m+\frac{1}{2}}U_{m}(x,0)dx\\
=&(-1)^{m-1}\int_{\mathbb{R}_{+}^{n+1}}|\nabla^{m+1} U_{m}|^{2}dxdy+(-1)^{m}\frac{\Gamma(m+1)\Gamma(\frac{1}{2})}{\Gamma(m+\frac{1}{2})}
\int_{\mathbb{R}^{n}}
U_{m}(x,0)(-\Delta_{x})^{m+\frac{1}{2}}U_{m}(x,0)dx.
\end{split}
\end{equation*}
Therefore, by Theorem \ref{be2},
we have
\begin{equation*}
\begin{split}
\int_{\mathbb{R}_{+}^{n+1}}|\nabla^{m+1} U_{m}(x,y)|^{2}&dxdy=\frac{\Gamma(m+1)\Gamma(\frac{1}{2})}{\Gamma(m+\frac{1}{2})}
\int_{\mathbb{R}^{n}}
U_{m}(x,0)(-\Delta_{x})^{m+\frac{1}{2}}U_{m}(x,0)dx\\
\geq &\frac{\Gamma(m+1)\Gamma(\frac{1}{2})}{\Gamma(m+\frac{1}{2})}\frac{\Gamma(\frac{n+2m+1}{2})}{\Gamma(\frac{n-2m-1}{2})}
\omega^{\frac{2m+1}{n}}_{n}\left(\int_{\mathbb{R}^{n}}|U_{m}(x,0)|^{\frac{2n}{n-2m-1}}dx\right)^{\frac{n-2m-1}{n}}.
\end{split}
\end{equation*}
With the same argument in Theorem \ref{th1.6} and using Theorem \ref{be2}, we have get the only
only extremal function is that given by (\ref{b1.23}). 
These complete the proof of Theorem \ref{th1.8}.

\end{document}